\documentclass[a4paper,11pt,reqno]{amsart}
\usepackage[english]{babel}
\usepackage{amsmath, amssymb, latexsym,mathptmx, amsfonts}

\usepackage[matrix,arrow,curve,frame]{xy}    

\xymatrixcolsep{1.9pc}                          
\xymatrixrowsep{1.9pc}
\newdir{ >}{{}*!/-5pt/\dir{>}}                  

\renewcommand{\subsection}{\@startsection{subsection}{1}{0pt}{-3.25ex plus -1ex minus-.2ex}{1.5ex plus.2ex}{\normalfont\it}}

 \calclayout

\DeclareMathOperator{\Hom}{Hom}
\DeclareMathOperator{\colim}{colim}
\DeclareMathOperator{\spec}{Spec}

\renewcommand{\leq}{\leqslant}
\renewcommand{\geq}{\geqslant}
\renewcommand{\phi}{\varphi}

\renewcommand{\kappa}{\varkappa}

\newtheorem{theorem}{Theorem}[section]

\newtheorem*{theorem*}{Theorem~\ref{ZM_fr_and_LM_fr}}

\newtheorem{lemma}[theorem]{Lemma}
\newtheorem{proposition}[theorem]{Proposition}
\newtheorem{corollary}[theorem]{Corollary}
\newtheorem{definition}[theorem]{Definition}
\newtheorem{notation}[theorem]{Notation}
\newtheorem{remark}[theorem]{Remark}
\newtheorem{example}[theorem]{Example}

\newtheorem{construction}[theorem]{Construction}

\newcommand{\bl}[1]{\buildrel #1\over}
\newcommand{\bb}{\mathbb}
\newcommand{\A}{\mathbb{A}}
\newcommand{\Spec}{\operatorname{Spec}}
\newcommand{\pt}{{\rm pt}}
\newcommand{\Gm}{{\mathbb{G}_m}}
\newcommand{\ZF}{\operatorname{\mathbb{Z}F}}
\newcommand{\Zf}{\operatorname{\mathbb{Z}f}}
\newcommand{\ZFr}{\operatorname{\mathbb{Z}Fr}}
\newcommand{\Fr}{\operatorname{Fr}}
\newcommand{\id}{\operatorname{id}}

\newcommand{\cc}{\mathcal}
\newcommand{\F}{\operatorname{F}}
\newcommand{\f}{\operatorname{f}}
\newcommand{\op}{{\textrm{\rm op}}}

\begin{document}
\sloppy

\title{Framed motives of relative motivic spheres}

\author{Grigory Garkusha}
\address{Department of Mathematics, Swansea University, Fabian Way, Swansea SA1 8EN, UK}
\email{g.garkusha@swansea.ac.uk}

\author{Alexander Neshitov}
\address{Department of Mathematics, University of Western Ontario,
London, Ontario N6A 5B7, Canada}
\email{aneshito@uwo.ca}

\author{Ivan Panin}
\address{St.~Petersburg Branch of V. A. Steklov Mathematical Institute,
Fontanka 27, 191023 St. Petersburg, Russia}


\email{paniniv@gmail.com}

\thanks{This paper was partly written during the visit of the second author to
Swansea University. He would like to thank the University for its
kind hospitality.}



\begin{abstract}
The category of framed correspondences $\Fr_*(k)$ and framed sheaves
were invented by Voevodsky in his unpublished notes~\cite{V2}. Based
on the theory, framed motives are introduced and studied
in~\cite{GP1}.
These are Nisnivich sheaves of $S^1$-spectra and the major computational tool of~\cite{GP1}.
The aim of this paper is to show the following result
which is essential in proving the main theorem of~\cite{GP1}: given an infinite perfect base field $k$,
any $k$-smooth scheme $X$ and any $n\geq 1$, the map of simplicial
pointed  Nisnevich sheaves $(-,\A^1//\mathbb G_m)^{\wedge n}_+\to T^n$ induces
a Nisnevich local level weak equivalence of $S^1$-spectra
   $$M_{fr}(X\times (\A^1// \mathbb G_m)^{\wedge n})\to M_{fr}(X\times T^n).$$
Moreover, it is proven that the sequence of $S^1$-spectra
   $$M_{fr}(X \times T^n \times \mathbb G_m) \to M_{fr}(X \times T^n \times\bb A^1) \to M_{fr}(X \times T^{n+1})$$
is locally a homotopy cofiber sequence in the Nisnevich topology.
Another important result of this paper shows that homology of framed
motives is computed as linear framed motives in the sense
of~\cite{GP1}. This computation is  crucial for the whole machinery
of framed motives~\cite{GP1}.
\end{abstract}

\maketitle

\thispagestyle{empty} \pagestyle{plain}


\section{Introduction}

Based on Voevodsky's theory of framed correspondences~\cite{V2}, the
machinery of framed motives has been introduced and developed in~\cite{GP1}.
This machinery leads to serious computations not only in motivic homotopy theory, but also in classical algebraic topology.
One of such computations~\cite[Theorem~11.9]{GP1} states that if $k$ is the field of
complex numbers, then the framed motive $M_{fr}(pt)(pt)$ of the point $pt = Spec (k)$ evaluated at $pt$
has the stable homotopy type of the classical sphere spectrum $\Sigma^{\infty}_{S^1}(S^0)$. In particular,
the sphere spectrum $\Sigma^{\infty}_{S^1}(S^0)$ can be computed in terms of algebraic varieties only.

The key ingredients for this computation are the theorem~\cite[Theorem~11.1]{GP1} computing the motivic sphere
spectrum in terms of twisted framed motives, the Cancellation Theorem for framed motives~\cite{AGP} as well as
a result by Levine~\cite[Corollary~2]{L}.
In turn, a key ingredient for proving~\cite[Theorem 11.1]{GP1} is this:
for each integer $n\geq 1$ the canonical morphism of motivic spaces
$C_*Fr((\bb A^1//\bb G_m)^{\wedge n}) \xrightarrow{}C_*Fr(T^{n})$
is a Nisnevich local equivalence. But this is exactly a partial case of the first statement
of Theorem~\ref{cone} proven in this paper.

The machinery of framed motives also produces explicit fibrant
resolutions for the suspension $\bb P^1$-spectrum $\Sigma^\infty_{\bb P^1}X_+$
of a $k$-smooth algebraic variety $X\in Sm/k$~\cite[Theorem~4.1]{GP1} or, more generally, for
the suspension $\bb P^1$-spectrum $\Sigma^\infty_{\bb P^1}\cc X$ of any pointed motivic space
$\cc X$~\cite[Section~10]{GP1}.
Theorem~\ref{cone} and its Corollary~\ref{cormain} proven in this paper play a key role
in the proof of \cite[Theorem~4.1]{GP1} as well as for results of~\cite[Section~10]{GP1}.

Indeed, one of the major steps in the proof of \cite[Theorem 4.1]{GP1}
is to show that for each integer $n\geq 1$ the canonical morphism
$C_*Fr(X\times T^n)_f\to\underline{\Hom}(\bb P^{\wedge 1},C_*Fr(X\times T^{n+1})_f)$
is a section-wise weak equivalence of motivic spaces (see~\cite[Section~9]{GP1} for details). Here $f$ stands for
a fibrant replacement within the injective local model structure.
Theorem~\ref{cone} and its Corollary~\ref{cormain} proven in this paper
are used in \cite[Section 9]{GP1} to reduce this major step to the following two statements.
The first one requires to show that for each
$Y\in Sm/k$ the canonical morphism
   $$M_{fr}(Y)_{\text{f}} \to \underline{\Hom}(\bb G_m^{\wedge 1},
       M_{fr}(Y\times\bb G_m^{\wedge 1})_{\text{f}})$$
is a sectionwise stable equivalence of $S^1$-spectra. Here $``\text{f}"$ refers to
a stable Nisnevich local fibrant replacement of $S^1$-spectra.
The second one requires to show that for each
$Z\in Sm/k$ the canonical morphism
$$M_{fr}(Z)\to \underline{\Hom}(S^1,M_{fr}(Z\otimes S^1))$$
is a sectionwise stable equivalence of $S^1$-spectra.
As explained in \cite[Section 9]{GP1} the first statement is nothing but the Cancellation
Theorem for framed motives~\cite[Theorem~A]{AGP}. The second statement follows from
the Additivity Theorem in~\cite[Theorem 6.5]{GP1}.

In order to formulate Theorem \ref{cone} below, let us briefly describe the relevant notions.
Let $\Fr_0(k)$ be the category whose objects are those of $Sm/k$ and
whose morphism set between $X$ and $Y$ is given by
the set of framed correspondences of level zero \cite[Example 2.1]{V2}, \cite[Definition 2.1]{GP1}.
As it is shown in~\cite[Section~5]{GP1}, the category of framed correspondences
of level zero $\Fr_0(k)$ has an action by finite pointed sets $U\otimes
K:=\bigsqcup_{K\setminus *}U$ with $U\in Sm/k$ and $K$ a finite pointed
set. The cone of $U$ is the simplicial object $U\otimes I$ in
$\Fr_0(k)$, where $(I,1)$ is the pointed simplicial set $\Delta[1]$
with basepoint 1. There is a natural morphism $i_0:U\to U\otimes I$
in $\Delta^{\textrm{\op}}\Fr_0(k)$. Given an inclusion of smooth
schemes $j: U\hookrightarrow Y$, denote by $Y//U$ a simplicial
object in $\Fr_0(k)$ which is obtained by taking the pushout of the diagram
   $Y\hookleftarrow U\bl{i_0}\hookrightarrow U\otimes I$
in $\Delta^{\textrm{\op}}\Fr_0(k)$.
The simplicial object $Y//U$ termwise equals
   $Y,Y\sqcup U,Y\sqcup U\sqcup U,\ldots$ .
Recall that the category $\Fr_0(k)$ is a full subcategory of
$SmOp(\Fr_0(k))$ (see~\cite[Section~5]{GP1} for definitions).
If $j: U\hookrightarrow Y$ is an open inclusion, then we have
an object
$(Y,U)\in SmOp(\Fr_0(k))$. Regard it as an object of the category
$\Delta^{\op}(SmOp(\Fr_0(k)))$ which is constant in the
simplicial direction. Then there is an obvious morphism
$\alpha: Y//U \to (Y,U)$ in the latter category. Namely, at each simplicial level $m$,
$\alpha$ is a morphism in $SmOp(\Fr_0(k))$
   $$\alpha_m: Y\sqcup U\sqcup \ldots \sqcup U \to (Y,U)$$
such that $\alpha_m|_Y=\id_Y: (Y,\emptyset)\to (Y,U)$,
$\alpha_m|_U=j: (U,\emptyset)\to (Y,U)$. The categories $\Fr_0(k)$
and $SmOp(\Fr_0(k))$ are both symmetric monoidal
(see~\cite[Section~5]{GP1}) and the inclusion of $\Fr_0(k)$
into $SmOp(\Fr_0(k))$ is a strict monoidal functor.
Thus we have the $n$th monoidal power $(Y//U)^{\wedge n}$ (respectively $(Y,U)^{\wedge n}$) of
$Y//U$ (respectively of $(Y,U)$) and a morphism $\alpha^{\wedge n}: (Y//U)^{\wedge n}\to (Y,U)^{\wedge n}$,
which is the $n$th monoidal power of $\alpha$.
For any $k$-smooth variety $X$ one has the morphism
$\id_X\times \alpha^{\wedge n}: X\times (Y//U)^{\wedge n}\to X\times (Y,U)^{\wedge n}$
in $\Delta^{\op}(SmOp(\Fr_0(k)))$. Applying the functor
$M_{fr}$ from~\cite[Definition~5.2]{GP1} to this morphism,
we get a morphism of framed $S^1$-spectra
   $$M_{fr}(X\times (Y//U)^{\wedge n}) \xrightarrow{M_{fr}(\id_X\times\alpha^{\wedge n})} M_{fr}(X\times (Y,U)^{\wedge n}).$$
In the special case when $(Y,U)=(\A^1,\bb G_m)$, the $S^1$-spectrum $M_{fr}(X\times (Y,U)^{\wedge n})$
will be denoted by $M_{fr}(X\times T^n)$. Note that $X\times(\A^1,\mathbb G_m)^{\wedge n}=(X\times\bb A^n,X\times(\bb A^n-0))$.


As explained above, the computation of an explicit motivically fibrant resolution of the suspension $\bb P^1$-spectrum
$\Sigma^\infty_{\bb P^1}X_+$, given in~\cite[4.1]{GP1}, requires
Theorem~\ref{cone} and its Corollary~\ref{cormain}.

\begin{theorem}\label{cone}
Let $k$ be an infinite perfect field.
For any $k$-smooth scheme
$X\in Sm/k$ and any $n\geq 1$,
the morphism $\alpha^{\wedge n}: (\A^1// \mathbb G_m)^{\wedge n}\to (\A^1,\mathbb G_m)^{\wedge n}$
of simplicial objects in $SmOp(\Fr_0(k))$
induces a level Nisnevich local weak equivalence of $S^1$-spectra
   $$M_{fr}(\id_X\times \alpha^{\wedge n}): M_{fr}(X\times (\A^1// \mathbb G_m)^{\wedge n})\to M_{fr}(X\times T^n).$$
Moreover, the sequence of $S^1$-spectra
   $$M_{fr}(X \times T^n \times \mathbb G_m) \to M_{fr}(X \times T^n \times\bb A^1) \to M_{fr}(X \times T^{n+1})$$
is locally a homotopy cofiber sequence in the Nisnevich topology.
\end{theorem}

The main goal of the paper is to prove this theorem and its consequence concerning
the natural morphism~\eqref{eq:main_l_w_equivalence} described below.
The first part of Theorem~\ref{cone} states that
the $S^1$-spectum $M_{fr}(X\times T^n)$ of the relative motivic sphere
$X\times \bb A^n/(X\times(\bb A^n-0))$ can locally be computed as
the $S^1$-spectrum $M_{fr}(X\times (\A^1//\mathbb G_m)^{\wedge n})$ of the
simplicial object $X\times(\A^1//\mathbb G_m)^{\wedge n}$
in $\Fr_0(k)$.
Another consequence of the theorem says that for every $n\geq 0$ the natural morphism
   \begin{equation}\label{eq:main_l_w_equivalence}
   M_{fr}(X\times T^n\times(\A^1//\mathbb G_m))\to M_{fr}(X\times T^{n+1})
   \end{equation}
is a level weak equivalence of $S^1$-spectra locally in the
Nisnevich topology (see Corollary~\ref{cormain}).
As explained above,
the proof of~\cite[4.1]{GP1} depends on the local
equivalence~\eqref{eq:main_l_w_equivalence}.

Theorem~\ref{cone} and its Corollary~\ref{cormain} concerning the local
equivalence~\eqref{eq:main_l_w_equivalence} are proved in Section~\ref{sec:cone}.

We have already discussed the importance of Theorem~\ref{cone}. Let us explain why
it is highly non-trivial. Indeed, using Linearisation Theorem~\ref{ZM_fr_and_LM_fr} stated below one can check
that the second assertion of Theorem~\ref{cone} is equivalent to the assertion
that the morphism $\tau$ from Theorem~\ref{th:Main} is a Nisnevich local equivalence.
For simplicity take $X=\pt$ and $n=0$. In this case
the domain of $\tau$ is the complex $C_*(\ZF(\A^1)/\ZF(\Gm))$ of linear framed presheaves.
The codomain of $\tau$ is the complex $C_*\ZF(T)$ of linear framed presheaves
and $\tau=C_*(p)$, where
   $$p: \ZF(\A^1)/\ZF(\Gm)\to \ZF(T)$$
is a natural morphism of presheaves. 
The morphism $p$ is far from being Nisnevich locally
an isomorphism, because otherwise Lemma~\ref{l:weak-eq_and_C_*} would imply that
$\tau=C_*(p)$ is locally a quasi-isomorphism of complexes.
However, it is not the case: the morphism of presheaves
   $$p: \ZF(\A^1)/\ZF(\Gm)\to \ZF(T)$$
is not locally an isomorphism (see Example~\ref{ZF_qf_is_not_ZF}).
This morphism is locally a monomorphism only. Its image is locally
$\ZF^{qf}(T)$ with $\ZF^{qf}(T)$ being a certain linear framed subpresheaf of $\ZF(T)$.
The presheaf $\ZF^{qf}(T)$ is of independent interest.
An advantage of $\ZF^{qf}(T)$ is that its sections are described by explicit geometric data.
Using properties of $\ZF^{qf}(T)$, Theorem~\ref{th:Main} then splits in two further
Theorems~\ref{p:main} and~\ref{p:moving} formulated below.
As we can see, Theorem \ref{cone} requires several non-trivial reductions showing
that the theorem itself is a highly non-trivial result.

In the rest of the introduction we describe the steps required for the proof of Theorem \ref{cone}
as well as fix some notation.
The following theorem (see Theorem~\ref{ZM_fr_and_LM_fr}) is crucial
in the analysis of framed motives. It allows to reduce many
computations for framed motives of algebraic varieties to analogous
computations for complexes of linear framed presheaves, which are
normally much simpler. In particular, the theorem computes homology of framed
motives. As an application of Theorem~\ref{ZM_fr_and_LM_fr},
Theorem~\ref{cone} reduces to Theorem~\ref{th:Main}. It is as well
worthwhile to mention another similar application of this kind.
In~\cite{AGP} the Cancellation Theorem for framed motives of
algebraic varieties is proved by reducing it to complexes of linear
framed presheaves.

We are now in a position to formulate the Linearisation Theorem.

\begin{theorem}[Linearisation]\label{ZM_fr_and_LM_fr}
For any integer $m\geq 0$, the natural morphism of framed
$S^1$-spectra
   $$\lambda_{X\times T^m}:\bb Z\Fr_*^{S^1}(X\times T^m)\to EM(\ZF_*(-,X\times T^m))$$
is a schemewise stable equivalence. Moreover,
the natural morphism of framed $S^1$-spectra
   $$l_{X\times T^m}: \bb ZM_{fr}(X\times T^m)\to LM_{fr}(X\times T^m)$$
is a schemewise stable equivalence. In particular, for any $U\in
Sm/k$ one has
   $$\pi_*(\bb ZM_{fr}(X\times T^m)(U))=H_*(\ZF(\Delta^\bullet\times U,X\times T^m))=H_*(C_*\bb Z\F(U,X\times T^m)).$$
\end{theorem}

Notation used in Theorem \ref{ZM_fr_and_LM_fr} is explained in Section~\ref{sec:cone}.
The theorem itself is proved in Appendix B. Thus in order to prove Theorem~\ref{cone},
it is sufficient to prove the following theorem.

\begin{theorem}\label{th:Main}
Let $k$ be an infinite perfect field. For any $k$-smooth scheme $X$
and any $n\geq 0$ the natural map of complexes of linear framed
presheaves
   $$\tau:C_*\ZF(X\times T^n \times \A^1)/C_*\ZF(X\times T^n \times \Gm)\to C_*\ZF(X\times T^{n+1})$$
is a Nisnevich local equivalence.
\end{theorem}

The proof of Theorem~\ref{th:Main} splits in two steps, each of
which is of independent interest. Firstly we introduce a linear
framed subpresheaf $\ZF^{qf}(X\times T^{n+1})$ of the linear
framed presheaf $\ZF(X\times T^{n+1})$ and prove the following

\begin{theorem}\label{p:main}
For any $k$-smooth scheme $X$ and any $n\geq 0$, the natural
morphism
   $$\ZF(X\times T^n\times \A^1)/\ZF(X\times T^n\times \Gm)\to \ZF^{qf}(X\times T^{n+1})$$
of linear framed presheaves is an isomorphism locally in the
Nisnevich topology.
\end{theorem}

By applying the Suslin complex, one can show that the natural
morphism of complexes of linear framed presheaves
   \begin{equation}\label{eq:factor_complex}
    \mu:C_*\ZF(X\times T^n\times \A^1)/C_*\ZF(X\times T^n\times \Gm)\to C_*\ZF^{qf}(X\times T^{n+1})
   \end{equation}
is locally a quasi-isomorphism in the Nisnevich topology (see
Proposition~\ref{p:factor_complex} for details).

Secondly using a moving lemma discussed in Section~\ref{mlemma}, we
then prove the following

\begin{theorem}\label{p:moving}
The inclusion of complexes of linear framed presheaves
$$C_*\ZF^{qf}(X\times T^{n+1}) \hookrightarrow C_*\ZF(X\times T^{n+1})$$
is locally a quasi-isomorphism in the Zariski topology.
\end{theorem}

Clearly, Theorem~\ref{p:moving} together with the Nisnevich local
quasi-isomorphism~\eqref{eq:factor_complex} imply
Theorem~\ref{th:Main}.

Throughout the paper we denote by $Sm/k$ the category of smooth
separated schemes of finite type over the base field $k$. The
subcategory of affine smooth $k$-varieties is denoted by $AffSm/k$.
Given a scheme $W$ and a family of regular function
$\phi_1,\ldots,\phi_m$ on $W$, we write $Z(\phi_1,\ldots,\phi_{m})$
to denote the closed subset in $W$ which is the common vanishing locus
of the family $\phi_1,\ldots,\phi_m$. Whenever we speak about $\Gamma$-spaces
we follow the terminology of Bousfield--Friedlander~\cite{BF}.
The category of pointed Nisnevich sheaves (respectively framed sheaves)
will be denoted by $Shv_\bullet(Sm/k)$ (respectively $Shv_\bullet^{fr}(Sm/k)$). We shall
also write $Pre_{\cc Ab}(Sm/k)$ to denote the category of presheaves of Abelian groups on $Sm/k$.
Whenever the authors write ``locally", it is always assumed ``locally in the Nisnevich topology".
They also stress that motivic equivalences are
not used anywhere in this paper except the proof of Lemma~\ref{l:weak-eq_and_C_*}.

\section{Framed presheaves $\Fr(-,Y/(Y-S))$ and $\ZF(-,Y/(Y-S))$}\label{sec:framed_Y_Y-S}

\begin{definition}\label{def:FrY/Y-S}{\rm
(I) Let $Y$ be a $k$-smooth scheme and $S\subset Y$ be a closed
subset and let $U\in Sm/k$. An {\it explicit framed correspondence
of level $m\geq 0$ from $U$ to $Y/(Y-S)$} consists of the following
tuples:
   $$(Z,W,\phi_1,\ldots,\phi_{m};g:W\to Y),$$
where $Z$ is a closed subset of $U\times\bb A^m$, finite over $U$,
$W$ is an \'{e}tale neighborhood of $Z$ in $U\times\bb A^m$,
$\phi_1,\ldots,\phi_{m}$ are regular functions on $W$, $g$ is a
regular map such that $Z=Z(\phi_1,\ldots,\phi_{m})\cap g^{-1}(S)$.
The set $Z$ is called the {\it support\/} of the explicit framed
correspondence. We shall also write quadruples $\Phi = (Z,W,\phi;g)$
to denote explicit framed correspondences.

(II) Two explicit framed correspondences $(Z,W,\phi;g)$ and
$(Z',W',\phi';g')$ of level $m$ are said to be {\it equivalent\/} if
$Z=Z'$ and there exists an \'{e}tale neighborhood $W''$ of $Z$ in
$W\times_{\bb A^m_U}W'$ such that $\phi\circ pr$ agrees with
$\phi'\circ pr'$ and the morphism $g\circ pr$ agrees with $g'\circ
pr'$ on $W''$.

(III) A {\it framed correspondence of level $m$ from $U$ to
$Y/(Y-S)$\/} is the equivalence class of an explicit framed
correspondence of level $m$ from $U$ to $Y/(Y-S)$. We write
$\Fr_m(U,Y/(Y-S))$ to denote the set of framed correspondences of
level $m$ from $U$ to $Y/(Y-S)$. We regard it as a pointed set whose
distinguished point is the class $0_{Y/(Y-S),m}$ of the explicit
correspondence $(Z,W,\phi;g)$ with $W=\emptyset$.

(IV) If $S=Y$ then the set $\Fr_m(U,Y/(Y-S))$ is denoted by
$\Fr_m(U,Y)$ and is called the {\it set of framed correspondences of
level $m$ from $U$ to $Y$}.

(V) Following Voevodsky~\cite{V2}, the {\it category of framed
correspondences\/} $\Fr_*(k)$ has objects those of $Sm/k$ and its
morphisms are the sets $\Fr_*(U,Y):=\bigsqcup_{m\geq 0}\Fr_m(U,Y)$,
$U,Y\in Sm/k$.  The subcategory of $\Fr_*(k)$ of framed correspondences
of level zero will be denoted by $\Fr_0(k)$.

(VI) A {\it framed presheaf\/} is just a contravariant functor from
the category $\Fr_*(k)$ to sets.

}\end{definition}

Let $X,Y$ and $S$ be $k$-smooth schemes and let
\begin{gather*}
\Psi=(Z',\A^k\times V\xleftarrow{(\alpha,\pi')}
W',\psi_1,\psi_2,\dots,\psi_k;g:W'\to U)\in \Fr_k(V,U)
\end{gather*}
be an explicit correspondence of level $k$ from $V$ to $U$ and let
\begin{gather*}
\Phi=(Z,\A^m\times U\xleftarrow{(\beta,\pi)}
W,\varphi_1,\varphi_2,\dots,\varphi_m;g':W\to Y)\in \Fr_m(U,Y/(Y-S))
\end{gather*}
be an explicit correspondence of level $m$ from $U$ to $Y/(Y-S)$. We
define $\Psi^{*}(\Phi)$ as an explicit correspondence of level $k+m$
from $V$ to $Y/(Y-S)$ as
\[
(Z\times_U Z',\A^{k+m}\times V\xleftarrow{(\alpha,\beta,\pi')} W'\times_U W,\psi_1,\psi_2,\dots,\psi_k,\varphi_1,\varphi_2,\dots,\varphi_m,,g'\circ pr_W)\in
\Fr_{k+m}(V,Y/(Y-S)).
\]
Clearly, the pull-back operation $(\Psi,\Phi)\mapsto \Psi^{*}(\Phi)$
of explicit correspondences respects the equivalence relation on
them. We get a pairing
   \begin{equation}\label{eq:compos}
    \Fr_k(V,U)\times \Fr_m(U,Y/(Y-S))\to \Fr_{k+m}(V,Y/(Y-S))
   \end{equation}
making $\Fr_*(U,Y/(Y-S)):=\bigsqcup_{m\geq 0} \Fr_m(-,Y/(Y-S))$ a
$\Fr_*(k)$-presheaf.

Let $X,Y,S$ and $T$ be smooth schemes. There is an \textit{external
product}
 \begin{equation}\label{eq:ext_product}
\Fr_m(U,Y/(Y-S))\times \Fr_n(\pt,\pt) \xrightarrow{-\boxtimes -}
\Fr_{m+n}(U,Y/(Y-S))
 \end{equation}
given by
$$((Z,W,\varphi_1,\varphi_2,\dots,\varphi_m;g),(Z',W',\psi_1,...,\psi_n))\mapsto
(Z\times Z',W\times
W',\varphi_1,\varphi_2,\dots,\varphi_m,\psi_1,...,\psi_n;g).$$

Set $\sigma:= (\{0\},\A^1,id : A^1 \to \A^1,const : \A^1 \to pt)\in \Fr_1(pt,pt)$.
Denote by
$$\Sigma: \Fr_m(U,Y/(Y-S))\to \Fr_{m+1}(U,Y/(Y-S))$$
the map $\Phi \mapsto \Phi\boxtimes \sigma$. Following
Voevodsky~\cite{V2} we give the following

\begin{definition}\label{def:Fr}{\rm
We shall refer to the set
\[
\Fr(U,Y/(Y-S)):=\colim(\Fr_0(U,Y/(Y-S))\xrightarrow{\Sigma}\Fr_1(U,Y/(Y-S))
\xrightarrow{\Sigma} \Fr_2(U,Y/(Y-S)) \dots)
\]
as the \textit{set stable framed correspondences from $U$ to
$Y/(Y-S)$}. Clearly, $\Fr(-,Y/(Y-S))$ is a framed presheaf of
pointed sets with the empty framed correspondence being the
distinguished point.

Clearly, $\Fr(-,Y/(Y-S))$ is even \textit{a framed functor} in the sense of
Voevodsky \cite{V2} meaning that $\Fr(\emptyset)=*$ and
$\Fr(U_1\sqcup U_2,Y/(Y-S))=\Fr(U_1,Y/(Y-S))\times
\Fr(U_2,Y/(Y-S))$.

}\end{definition}

\begin{definition}[cf.~\cite{GP1}]\label{stab}{\rm
 Let $Y\in Sm/k$ and $S\subset Y$ be as in Definition~\ref{def:FrY/Y-S}. Let $U$ be a $k$-smooth scheme. Denote by
\begin{itemize}
\item[$\diamond$]
$\ZFr_m(U,Y/(Y-S)):=\widetilde{\mathbb{Z}}[\Fr_m(U,Y/(Y-S))]=\mathbb{Z}[\Fr_m(U,Y/(Y-S))]/\mathbb{Z}\cdot 0_{Y/(Y-S),m}$,
i.e the free abelian group generated by the set $\Fr_m(U,Y/(Y-S))$
modulo $\mathbb{Z}\cdot 0_{Y/(Y-S),m}$;

\item[$\diamond$]
$\ZF_m(U,Y/(Y-S)):=\ZFr_m(U,Y/(Y-S))/A$, where $A$ is the subgroup
generated by the elements
\begin{multline*}
(Z\sqcup Z', W,(\varphi_1,\varphi_2,\dots,\varphi_m);g) - \\
-(Z, W\setminus
Z',(\varphi_1,\varphi_2,\dots,\varphi_m)|_{W\setminus
Z'};g|_{W\setminus Z'}) - (Z',{W\setminus
Z},(\varphi_1,\varphi_2,\dots,\varphi_m)|_{W\setminus
Z};g|_{W\setminus Z}).
\end{multline*}
\end{itemize}
The elements of $\ZF_m(U,Y/(Y-S))$ are called {\it linear framed
correspondences from $U$ to $Y/(Y-S)$ of level $m$}.
}
\end{definition}

\begin{definition}\label{F_m_U_Y/(Y-S)}{\rm
Define $\F_m(U,Y/(Y-S))\subset \Fr_m(U,Y/(Y-S))$ as the subset consisting of
$(Z,W,\phi;g)\in \Fr_m(U,Y/(Y-S))$ such that $Z$ is connected.

Clearly, the set $\F_m(U,Y/(Y-S))-0_m$ is a free basis of the abelian group
$\ZF_m(U,Y/(Y-S))$. However, the assignment $U\mapsto \F_m(U,Y/(Y-S))$
is not a presheaf even on the category $Sm/k$.

Indeed, if $u\xrightarrow{i} U$ is a closed point and $(Z,W,\phi;g)\in \F _m(U,Y/(Y-S))\subseteq \Fr_m(U,Y/(Y-S))$,
then the support of $i^*(Z,W,\phi;g)\in \Fr_m(u,Y/(Y-S))$ is the closed subsetset $Z_u$ in $\A^m_u$.
Clearly, $Z_u$ is not connected in general.
}\end{definition}

The category of {\it linear framed correspondences $\ZF_*(k)$\/} is
defined in~\cite{GP1}. We shall also refer to contravariant functors
from the category $\ZF_*(k)$ to Abelian groups as {\it linear framed
presheaves}.

Set $\ZF_*(U,Y/(Y-S))=\bigoplus_{m\geq 0}\ZF_m(U,Y/(Y-S))$. The
pairing~\eqref{eq:compos} induces in a natural way a bilinear
pairing
\begin{equation}\label{eq:linear_comp}
\ZF_k(V,U)\times \ZF_m(U,Y/(Y-S))\to \ZF_{k+m}(U,Y/(Y-S)).
\end{equation}
The latter pairing makes $\ZF_*(-,Y/(Y-S))$ a linear framed
presheaf.

The external product~\eqref{eq:ext_product} induces in a natural way
an external product of the form
\begin{equation}\label{eq:linear_ext_product}
\ZF_m(U,Y/(Y-S)))\times \ZF_n(\pt,\pt) \xrightarrow{-\boxtimes -}
\ZF_{m+n}(U,Y/(Y-S))
\end{equation}

Let $Y\in Sm/k$ and $S\subset Y$ be as in
Definition~\ref{def:FrY/Y-S}. {\it One of the main linear framed
presheaves of this paper\/} we are interested in is defined as
$$\ZF(-,Y/(Y-S))=\colim(\ZF_0(-,Y/(Y-S))\xrightarrow{\Sigma}\ZF_1(-,Y/(Y-S))\xrightarrow{\Sigma} \ZF_2(-,Y/(Y-S))\xrightarrow{\Sigma}\dots).$$


\begin{definition}\label{def:SmOpFr_0}{\rm
Following~\cite[Definition~5.1]{GP1},
we define a category $SmOp(\Fr_0(k))$, which will often be used in
our constructions. Its objects are pairs $(X,U)$, where $X\in Sm/k$
and $U\subset X$ is an open subset. A morphism between $(X,U)$ and
$(X',U')$ in $SmOp(\Fr_0(k))$ is a morphism $f\in \Fr_0(X,X')$ such
that $f(U)\subset U'$. We shall also identify $X\in Sm/k$ with the
pair $(X,\emptyset)\in SmOp(\Fr_0(k))$.

}\end{definition}

The category $SmOp(\Fr_0(k))$ is symmetric monoidal with the monoidal
product $\wedge$ given by
   $$(X,U)\wedge (Y,V):=(X\times Y,X\times V\cup U\times Y).$$
The point $\pt$ is its monoidal unit. If $U=X-S$ and $V=Y-T$ then
   $$(X,X-S)\wedge (Y,Y-T)=(X\times Y,X\times(Y-T)\cup (X-S)\times Y)=(X\times Y,X\times Y-S\times T).$$
If $X\in Sm/k$ we will also write $X\times(Y,V)$ (respectively, $(Y,V)\times X$) to denote the object
$(X,\emptyset)\wedge(Y,V)=(X\times Y,X\times V)$ (respectively,
$(Y,V)\wedge(X,\emptyset)=(Y\times X,V\times X)$) of $SmOp(\Fr_0(k))$.
Also, by $(X,U)\sqcup (Y,V)$ we shall mean $(X \sqcup Y, U \sqcup V)$.

Define now covariant functors $\Fr_m:SmOp(\Fr_0(k))\to Shv_\bullet(Sm/k)$
and $\Fr:SmOp(\Fr_0(k))\to Shv^{fr}_\bullet(Sm/k)$.

\begin{construction}[The covariant functor $\Fr_m$]\label{Pairs_to_Shv}{\rm
For any object $(Y,Y-S)$ in $SmOp(\Fr_0(k))$ and any integer $m\geq 0$,
the value of $\Fr_m$ at $(Y,Y-S)$ is the sheaf $\Fr_m(-,Y/(Y-S))$.
For each $\f\in \Hom_{SmOp(\Fr_0(k))}((Y,Y-S),(Y',Y'-S'))$, write
$\f_{*,m}: \Fr_m(-,Y/(Y-S))\to \Fr_m(-,Y'/(Y'-S'))$ to denote the following morphism of sheaves:
given $U\in Sm/k$, it takes
$(Z,W,\phi_1,\ldots,\phi_{m};g)\in \Fr_m(U,Y/(Y-S))$ to
$(Z(\phi_1,\ldots,\phi_{m})\cap (\f\circ g)^{-1}(S'),W,\phi_1,\ldots,\phi_{m};\f\circ g)\in \Fr_m(-,Y'/(Y'-S'))$.
Clearly, the assignments $(Y,Y-S)\mapsto\Fr_m(-,Y/(Y-S))$ and $\f\mapsto \f_{*,m}$ form
a functor
   $$\Fr_m:SmOp(\Fr_0(k))\to Shv_\bullet(Sm/k).$$
Likewise, the assignments $(Y,Y-S)\mapsto\bb Z[\Fr_m(-,Y/(Y-S))]$ and $\f\mapsto \bb Z[\f_{*,m}]$ form a functor
   $$\bb Z[\Fr_m]: SmOp(\Fr_0(k)) \to Pre_{\cc Ab}(Sm/k)$$
as well as the assignments $(Y,Y-S)\mapsto\ZF_m(-,Y/(Y-S))]$ and $\f\mapsto \Zf_{*,m}$ form a functor
$$\ZF_m: SmOp(\Fr_0(k))\to Pre_{\cc Ab}(Sm/k).$$
}
\end{construction}

\begin{construction}[The covariant functor $\Fr$]\label{Pairs_to_fr_Shv}{\rm
For each integer $m\geq 0$ and each object $(Y,Y-S)$ in $SmOp(\Fr_0(k))$,
write $\Sigma_{m,Y/(Y-S)}$ for the morphism of sheaves $\Sigma: \Fr_m(-,Y/(Y-S))\to \Fr_{m+1}(-,Y/(Y-S))$
given in Definition \ref{def:Fr}.
It is easy to check that the assignment
$(Y,Y-S)\mapsto \Sigma_{m,Y/(Y-S)}$
gives rise to a natural transformation of functors
$\Sigma: \Fr_m\ \to \Fr_{m+1}$. Taking the colimit
$$\Fr=\colim(\Fr_0\xrightarrow{\Sigma}\Fr_1
\xrightarrow{\Sigma} \Fr_2\xrightarrow{\Sigma} \dots),$$
we get a functor from $SmOp(\Fr_0(k))$ to $Shv_\bullet(Sm/k)$.
As mentioned in Definition
\ref{def:Fr} for each object $(Y,Y-S)\in SmOp(\Fr_0(k))$ the sheaf $\Fr(-,Y/(Y-S))$ is naturally a framed Nisnevich sheaf.
Moreover, for each morphism
$f: (Y,Y-S)\to (Y',Y'-S')$ in $SmOp(\Fr_0(k))$ the morphism $\Fr(f): \Fr(-,Y/(Y-S))\to \Fr(-,Y'/(Y'-S'))$
is a morphism of framed Nisnevich sheaves.

The covariant functors $(Y,Y-S)\mapsto\bb Z[\Fr(-,Y/(Y-S))]$ and
$(Y,Y-S)\mapsto\ZF(-,Y/(Y-S))]$ from $SmOp(\Fr_0(k))$ to $Pre_{\cc Ab}^{fr}(Sm/k)$
are defined in a similar fashion.
}\end{construction}

\begin{notation}\label{para}{\rm
Given $X\in Sm/k$ and $n>0$, we will often write $\Fr_m(X\times T^n)$ (respectively, $\Fr(X\times T^n)$,
$\bb Z[\Fr_m](X\times T^n)$, $\bb Z[\Fr](X\times T^n)$, $\bb ZF_m(X\times T^n)$, $\bb ZF(X\times T^n)$) to denote the value of the functor
$\Fr_m$ (respectively, $\Fr$, $\bb Z[\Fr_m]$, $\bb Z[\Fr]$, $\bb ZF_m$, $\bb ZF$)
on the couple $X\times(\bb A^1,\bb G_m)^{\wedge n}\in SmOp(\Fr_0(k))$.
Their sections on $U\in Sm/k$ will also be denoted by
$\Fr_m(U,X\times T^n)$, $\Fr(U,X\times T^n)$,
$\bb Z[\Fr_m](U,X\times T^n)$, $\bb ZF_m(U,X\times T^n)$,
$\bb Z[\Fr](U,X\times T^n)$, $\bb ZF(U,X\times T^n)$ respectively. Notice that
$X\times(\bb A^1,\bb G_m)^{\wedge n}=(X\times\bb A^n,X\times(\bb A^n-0))$.
}\end{notation}

There is an obvious functor $spc:
SmOp(\Fr_0(k)) \to Shv_\bullet(Sm/k)$ sending an object $(X,U)\in
SmOp(\Fr_0(k))$ to the pointed Nisnevich sheaf $X/U$. Observe that this
functor preserves the monoidal product.

Recall from~\cite[Section~3]{GP1} that for any $n\geq 0$,
the endofunctor $\cc Fr_n:Shv_\bullet(Sm/k)\to Shv_\bullet(Sm/k)$ takes $\cc F\in Shv_\bullet(Sm/k)$
to $\underline{\Hom}_{Shv_\bullet(Sm/k)}(\bb P^{\wedge n},\cc F\wedge T^n)$.
The stabilization in the canonical morphism $\sigma:\bb P^{\wedge 1}\to T$ leads to the endofunctor
$\cc Fr:Shv_\bullet(Sm/k)\to Shv_\bullet(Sm/k)$.
It follows from ~\cite[Section~3]{GP1} that there are canonical isomorphisms between functors
   $$a_n:\Fr_n\xrightarrow{\cong}\cc Fr_n\circ spc\quad{\textrm{and}}\quad a:\Fr\xrightarrow{\cong}\cc Fr\circ spc$$
in the category of functors from $SmOp(\Fr_0(k))$ to $Shv_\bullet(Sm/k))$.

\section{Presheaves $\Fr^{qf}(-,Y/(Y-S)\wedge T)$ and $\ZF^{qf}(-,Y/(Y-S)\wedge T)$}\label{sec:Fr_qf_and_ZF_qf}

One of the main objectives of Sections~\ref{sec:Fr_qf_and_ZF_qf}, \ref{sec:geom lemmas} and \ref{s:Fr(X_T_n}
is to give an explicit description of the quotient presheaf
$\ZF(-,X\times T^{n}\times \A^1)/\ZF(-,X\times T^{n}\times \mathbb G_m)$ on essentially
smooth local henselian schemes. It turns out
that this quotient presheaf coincides on such schemes with a subpresheaf
$\ZF^{qf}(-,X\times T^{n+1})$ of the presheaf $\ZF(-,X\times T^{n+1})$ introduced in this section.
More precisely, the canonical morphism $p: (\A^1,\emptyset)\to (\bb A^1,\bb G_m)$ in $SmOp(\Fr_0k)$
induces a morphism of presheaves
   $$\bar p:\ZF(-,X\times T^{n}\times \A^1)/\ZF(-,X\times T^{n}\times {\mathbb G}_m) \to \ZF(-,X\times T^{n+1}).$$
Lemma \ref{l:ZF_qf_m_and_ZF_m}(2), Lemma \ref{l:pushout} and Corollary \ref{cor:pushout} imply that
the morphism $\bar p$ is locally a monomorphism and its image
coincides with the subpresheaf $\ZF^{qf}(-,X\times T^{n+1})$
of the presheaf $\ZF(-,X\times T^{n+1})$.
Examples given in this section show that
sections of this subpresheaf do not coincide with sections of
$\ZF(-,X\times T^{n+1})$
(even on the base field).
Thus the morphism $\bar p$
is far from being a local isomorphism.

\begin{notation}\label{rem: some equalities}{\rm
As recalled in Definition \ref{def:SmOpFr_0} the category $SmOp(\Fr_0(k))$ is symmetric monoidal.
Let $(Y,Y-S)\in SmOp(\Fr_0(k))$. Similarly to Notation~\ref{para} we shall write
$\Fr_m(-,Y/(Y-S)\wedge T)$ (respectively $\Fr(-,Y/(Y-S)\wedge T)$, $\ZF_m(-,Y/(Y-S)\wedge T)$,
$\ZF(-,Y/(Y-S)\wedge T)$) to denote the value of the functor $\Fr_m$ (respectively $\Fr$, $\ZF_m$, $\ZF$) on the couple
$(Y,Y-S)\wedge(\bb A^1,\bb G_m)=(Y\times \A^1,Y\times \A^1-S\times \{0\})$.
}\end{notation}

Following Notation~\ref{rem: some equalities}, a section of
$\Fr_m(-,Y/(Y-S)\wedge T)$ on $U\in Sm/k$ is given by a tuple
   $$(Z,W,\phi_1,\ldots,\phi_{m};g:W\to Y;f:W\to \A^1),$$
where $Z$ is a closed subset of $U\times\bb A^m$ finite over $U$,
$W$ is an \'{e}tale neighborhood of $Z$ in $U\times\bb A^m$,
$\phi_1,\ldots,\phi_{m},f$ are regular functions on $W$, $g$ is a
regular map such that
   $$Z=Z(\phi_1,\ldots,\phi_{m},f)\cap g^{-1}(S).$$
The suspension map
   $$\Sigma: \Fr_m(-,Y/(Y-S)\wedge T)\to\Fr_{m+1}(-,Y/(Y-S)\wedge T)$$
evaluated on $U\in Sm/k$ sends
$$(Z,W,\phi_1,\ldots,\phi_{m};g;f) \ \ \textrm{to} \ \ (Z\times \{0\},W\times \A^1,\phi_1\circ pr_W,\ldots,\phi_{m}\circ pr_W,pr_{\A^1};g\circ pr_W;f\circ pr_W).$$
For brevity, we write $(Z\times \{0\},W\times
\A^1,\phi,t;g;f)$ for the latter framed correspondence.

\begin{definition}\label{paraqf}{\rm
Let $\Fr^{qf}_m(U,Y/(Y-S)\wedge T)\subset \Fr_m(U,Y/(Y-S)\wedge T)$
be the subset consisting of those elements $c$ for which there is an
explicit framed correspondence $(Z,W,\phi_1,\ldots,\phi_{m};g;f)$
representing $c$ such that the closed subset
$Z(\phi_1,\ldots,\phi_{m})\cap g^{-1}(S)\subset W$ is
quasi-finite over $U$.
}\end{definition}

\begin{example}{\rm
Here are some simple examples of non quasi-finite framed correspondences\label{ZF_qf_is_not_ZF}:

$\bullet$ \ $(\{0\},\A^m,t_1,t_1,...,t_{m-1},\A^m\to \Spec (k);t_m)\in \Fr_m(\Spec (k),T)\setminus \Fr^{qf}_m(\Spec (k),T)$;

$\bullet$ \ $(\{0\},\A^m,t_1,t_2,t_2,..,t_{m-1},\A^m\to \Spec (k);t_m)\in \Fr_m(\Spec (k),T)\setminus \Fr^{qf}_m(\Spec (k),T)$;

$\bullet$ \ $(\{0\},\A^m,t_1,t_2,...,t_{m-1},t_{m-1},\A^m\to \Spec (k);t_m)\in \Fr_m(\Spec (k),T)\setminus \Fr^{qf}_m(\Spec (k),T)$.

More generally, let $\phi_1,\ldots,\phi_{m-1}\in k[\A^m]=k[t_1,...,t_m]$ be polynomials such that their common zero locus
in $\A^m$ is an equidimensional closed subset of dimension 1. Let $f\in k[t_1,...,t_m]$ be such that
the common zero locus $Z$ of the functions $\phi_1,\ldots,\phi_{m-1},f$ has dimension zero. Then

$\bullet$ \ $(Z,\A^m,\phi_1,\ldots,\phi_{m-1},\A^m\to \Spec (k);f)\in \Fr_m(\Spec (k),T)\setminus \Fr^{qf}_m(\Spec (k),T)$.
}\end{example}

\begin{definition}\label{def:Fr_qf}{\rm
Set $\Fr^{qf}_*(U,Y/(Y-S)\wedge T):=\bigsqcup_{m\geq
0}\Fr^{qf}_m(U,Y/(Y-S)\wedge T)$. Clearly,
$\Fr^{qf}_*(-,Y/(Y-S)\wedge T)$ is a framed subpresheaf of the
framed presheaf $\Fr_*(-,Y/(Y-S)\wedge T)$. Also, it is clear that
the suspension $\Sigma$ takes $\Fr^{qf}_m(U,Y/(Y-S)\wedge T)$ to
$\Fr^{qf}_{m+1}(U,Y/(Y-S)\wedge T)$. Set,
   $$\Fr^{qf}(-,Y/(Y-S)\wedge T):=\colim(\Fr^{qf}_0(-,Y/(Y-S)\wedge T)\xrightarrow{\Sigma}\Fr^{qf}_1(-,Y/(Y-S)\wedge T)\xrightarrow{\Sigma}\dots).$$
By the very construction, $\Fr^{qf}(-,Y/(Y-S)\wedge T)$ is a pointed
framed subpresheaf of the pointed framed presheaf
$\Fr(-,Y/(Y-S)\wedge T)$.
}
\end{definition}

Let $(Z,W,\varphi;g;f)\in \Fr_m(U,Y/(Y-S))$ be such that $Z=Z_1\sqcup Z_2$. Then
$(Z,W,\varphi;g;f)\in \Fr^{qf}_m(U,Y/(Y-S))$ if and only if for $i,j=1,2$ and $j\neq i$ one has
   $$(Z_i, W\setminus Z_j,\varphi|_{W\setminus Z_j};g|_{W\setminus Z_j};f|_{W\setminus Z_j})\in \Fr^{qf}_m(U,Y/(Y-S)).$$

This observation leads to the following

\begin{definition}\label{def:ZF_m_qf}{\rm
Let $Y\in Sm/k$ and $S\subset Y$ be as in
Definition~\ref{def:FrY/Y-S}. Let $U$ be a $k$-smooth scheme. Set,
   $$\ZF^{qf}_m(U,Y/(Y-S)\wedge T):=\bb Z[\Fr^{qf}_m(U,Y/(Y-S)\wedge T)]/A,$$
where $A$ is the subgroup generated by the elements
$$(Z\sqcup Z', W,\varphi;g;f)
-(Z, W\setminus Z',\varphi|_{W\setminus Z'};g|_{W\setminus
Z'};f|_{W\setminus Z'}) - (Z',{W\setminus Z},\varphi|_{W\setminus
Z};g|_{W\setminus Z};g|_{W\setminus Z}).$$

}\end{definition}

Set $\ZF^{qf}_*(U,Y/(Y-S)\wedge T)=\bigoplus_{m\geq
0}\ZF^{qf}_m(U,Y/(Y-S)\wedge T)$. The pairing~\eqref{eq:linear_comp}
gives rise to a natural pairing $\ZF_k(V,U)\times
\ZF^{qf}_m(U,Y/(Y-S))\to \ZF^{qf}_{k+m}(U,Y/(Y-S))$. The latter
pairing makes $\ZF^{qf}_*(-,Y/(Y-S))$ a linear framed presheaf.

The external product~\eqref{eq:linear_ext_product} gives rise to an
external product
\begin{equation*}\label{potom}
\ZF^{qf}_m(U,Y/(Y-S))\wedge T)\times \ZF_n(\pt,\pt) \xrightarrow{-\boxtimes -}
\ZF^{qf}_{m+n}(U,Y/(Y-S)\wedge T).
\end{equation*}

\begin{definition}\label{def:ZF_qf}{\rm
Set,
   $$\ZF^{qf}(-,Y/(Y-S)\wedge T)=\colim(\ZF^{qf}_0(-,Y/(Y-S)\wedge T)\xrightarrow{\Sigma}\ZF^{qf}_1(-,Y/(Y-S)\wedge T)\xrightarrow{\Sigma}\dots).$$
By the very construction $\ZF^{qf}(-,Y/(Y-S)\wedge T)$ is a linear
framed presheaf.

}\end{definition}

\begin{definition}\label{def:F_m_and_F_m_qf}{\rm
For $U\in Sm/k$ let $\F_m(U,Y/(Y-S)\wedge T)$ be the subset of $\Fr_m(U,Y/(Y-S)\wedge T)$ consisting of the elements
$$(Z,W,\phi_1,\ldots,\phi_{m};g;f) \in \Fr_m(U,Y/(Y-S)\wedge T)$$
such that $Z$ is connected. Clearly, this definition is consistent with Definition~\ref{F_m_U_Y/(Y-S)}. Set,
   $$\F^{qf}_m(U,Y/(Y-S)\wedge T)=\F_m(U,Y/(Y-S)\wedge T)\cap \Fr^{qf}_m(U,Y/(Y-S)\wedge T).$$
$\F^{qf}_m(U,Y/(Y-S)\wedge T)\setminus {*}$ is plainly a free basis of the free abelian group
$\ZF^{qf}_m(U,Y/(Y-S)\wedge T)$.

}\end{definition}

The following lemma immediately follows from definitions.

\begin{lemma}\label{l:ZF_qf_m_and_ZF_m}
$(1)$ The set $(\F^{qf}_m(U,Y/(Y-S)\wedge T)\setminus{*})$ is a free
basis of $\ZF^{qf}_m(U,Y/(Y-S)\wedge T)$.

$(2)$ The natural map $\ZF^{qf}_m(U,Y/(Y-S)\wedge T) \to
\ZF_m(U,Y/(Y-S)\wedge T)$ is injective and identifies
$\ZF^{qf}_m(U,Y/(Y-S)\wedge T)$ with a direct summand of
$\ZF_m(U,Y/(Y-S)\wedge T)$. Therefore $\ZF^{qf}(-,Y/(Y-S)\wedge T)$
is a framed subpresheaf of the framed presheaf $\ZF(-,Y/(Y-S)\wedge T)$.
\end{lemma}

\begin{remark}{\rm
However, the assignment $U\mapsto \F^{qf}_m(U,Y/(Y-S)\wedge T)$
is not a presheaf even on the category $Sm/k$.
Indeed, if $u\xrightarrow{i} U$ is a closed point and
$$(Z,W,\phi;g)\in \F^{qf}_m(U,Y/(Y-S)\wedge T)\subseteq \Fr^{qf}_m(U,Y/(Y-S)\wedge T),$$
then the support of $i^*(Z,W,\phi;g)\in \Fr^{qf}_m(u,Y/(Y-S)\wedge T)$ is the closed subset $Z_u$ in $\A^m_u$,
but $Z_u$ is often not connected.
}\end{remark}


\section{Geometric lemmas}\label{sec:geom lemmas}

In this section we prove a couple of geometric lemmas used in the
following sections. All schemes in this section are supposed to be affine and noetherian.
If $X$ is an affine noetherian scheme and $Z\subset X$ is a closed subset then
by~\cite[6.9]{G}
the henselization $X^h_Z$ of $X$ at $Z$ is an affine noetherian scheme.
We start with a useful remark.

\begin{remark}\label{r:henselian_properties}{\rm
Let $W$ be a reduced local scheme
with a closed point $w$. Let
$S\subset W$ be a closed subset. Then $S$ with the reduced scheme
structure is a local scheme which is connected and $w$ is the only
closed point of $S$.

Furthermore, let $U$ be a reduced henselian local scheme
with the
closed point $u$.
Let $S\subset W$ be a closed subset.
Let $\pi: W\to U$ be a morphism such that
$\pi(w)=u$ and let $\pi|_S: S\to U$ be quasi-finite.
By~\cite[Theorem I.4.2]{Mi} $S$ is finite over $U$. Thus $S$ is reduced
henselian local.

Let $S$ be a henselian local scheme and let $Z$ be a closed
subset of $S$. Suppose $S^h_Z$ is the henselization of $S$ at $Z$.
Then the
canonical morphism $can_{S,Z}:S^h_Z \to S$ is an isomorphism.
Indeed, $(S,id_S,i:Z\hookrightarrow S)$ is the initial object in the
category of \'{e}tale neighborhoods of $Z$ in $S$.
Since $can_{S,Z}:S^h_Z \to S$ is an isomorphism, the scheme
$S^h_Z$ is noetherian henselian local.
}\end{remark}

\begin{lemma}\label{l:closed-embedding}
Let $V$ be an affine scheme, $Z\subset V$ be a closed connected
subset, $can=can_{V,Z}:V^h_Z \to V$ be the henselization of $V$ at
$Z$, and let $s:Z\to V^h_Z$ be the section of $can$ over $Z$. Let
$U$ be a regular local henselian scheme, $q:V\to U$ be a smooth
morphism such that the morphism $q|_Z:Z\to U$ is finite.
Furthermore, suppose $Y\subset V^h_Z$ is a closed subset containing
$s(Z)$, which is quasi-finite over $U$. Then $can|_Y: Y\to V$ is a
closed embedding, $can(Y)$ contains $Z$ and $can^{-1}(can(Y))=Y$.
\end{lemma}

\begin{proof}
Since $Z$ is finite over the local henselian $U$ and $Z$ is
connected, it is local and henselian. Since $Z$ is local, then so is
the scheme $W=V^h_Z$. By Remark~\ref{r:henselian_properties} the
scheme $Y$ with the reduced scheme structure
is local, connected and finite over $U$. Hence the set
$Y_1=can(Y)$ is closed in $V$ and finite over $U$ and contains $Z$.
So $Y_1$ with the reduced scheme structure is a local henselian scheme.

We claim that there is a canonical isomorphism of schemes between the henselization
$(Y_1)^h_Z$ of $Y_1$ at $Z$
and the scheme
$can^{-1}(Y_1)$. This claim is a partial case of \cite[Lemma 5.8]{FP}.
In order to apply that lemma,
our scheme $V$ is replaced by $S$, our scheme $Y_1$ is replaced by $S'$,
and our local scheme $Z$ is replaced by $T$.
Within this notation the scheme $T\times_S S'$ from that lemma is just our scheme $Z$.
Also, in this case the canonical scheme isomorphism
$(T\times_S S')^h\to T^h\times_S S'$ from that lemma
is the isomorphism of schemes
$(Y_1)^h_Z\to can^{-1}(Y_1)$.
Hence the pair $(can^{-1}(Y_1), s(Z))$ is the henselization
$(Y_1)^h_Z$ of $Y_1$ at $Z$.
Thus, by Remark~\ref{r:henselian_properties} the morphism
$p_1=can|_{can^{-1}(Y_1)}:can^{-1}(Y_1)\to Y_1$ is an isomorphism of schemes.
Since the scheme $Y_1$ is reduced, so is the scheme $can^{-1}(Y_1)$.

Clearly, $Y$ regarded with the reduced scheme structure is a closed subscheme
of the scheme $can^{-1}(Y_1)$. Let $i:Y \to
can^{-1}(Y_1)$ be the inclusion. Then the ring map $(p_1\circ i)^*:
\Gamma(Y_1,\mathcal O_{Y_1})\to \Gamma(Y,\mathcal O_{Y})$ is
surjective. On the other hand, this ring morphism is injective,
because both schemes are reduced affine and the morphism $p_1\circ
i$ is surjective. Hence $p_1\circ i:Y\to Y_1$ is an isomorphism of schemes.
Thus $can|_Y:Y\to V$ is a closed embedding.
\end{proof}

\begin{lemma}\label{l:lift}
Under the assumptions of Lemma~\ref{l:closed-embedding} supoose
$Y\subset V$ is a closed connected subset containing $Z$ and is
finite over $U$. Then there is a unique section $t:Y\to V^h_Z$ of
the morphism $can:V^h_Z\to V$ and $t(Y)=can^{-1}(Y)$ contains $s(Z)$.
\end{lemma}

\begin{proof}
Consider the closed subset $Y\subset V$ with the reduced scheme structure.
Since $U$ is local henselian, then so is the scheme $Y$.
Using~\cite[Lemma 5.8]{FP}
similarly to the proof of Lemma~\ref{l:closed-embedding},
we conclude that the scheme
$can^{-1}(Y)$ is the henselization of $Y$ at $Z$.
By Remark~\ref{r:henselian_properties} the morphism
$p=can|_{can^{-1}(Y)}:can^{-1}(Y)\to Y$ is
an isomorphism of schemes. Set
$t=in\circ p^{-1}:Y\to V^h_Z$,
where $in: can^{-1}(Y)\hookrightarrow V^h_Z$ is the closed embedding.
Clearly, $t$ is a section of $can$
over $Y$. If $t':Y\to V^h_Z$ is another section of $can$, then $t'=t$.
\end{proof}

\section{The linear framed presheaf $\ZF^{qf}(X\times T^{n+1})$}\label{s:Fr(X_T_n}

Let $X$ be a $k$-smooth variety and let $Y =X\times\A^n$, $S=X\times\{0\}$. Following Notations~\ref{para},~\ref{rem: some equalities}
and Definitions~\ref{paraqf},~\ref{def:ZF_m_qf},~\ref{def:ZF_qf}, we shall write $\Fr^{qf}(-,X\times T^{n+1})$ and
$\ZF^{qf}(-,X\times T^{n+1})$ to denote the presheaves $\Fr^{qf}(-,(Y/(Y-S))\wedge T)$ and
$\ZF^{qf}(-,(Y/(Y-S))\wedge T)$ respectively. We also have the following equalities:
   $$\ZF(-,X\times T^{n+1})=\ZF(-,(Y/(Y-S))\wedge T) \quad \text{and} \quad \Fr(-,X\times T^{n+1})=\Fr(-,(Y/(Y-S))\wedge T).$$
We also have equalities of framed presheaves as follows:
\begin{itemize}
\item[$\diamond$] $\Fr^{qf}_*(-,X\times T^{n+1})=\Fr^{qf}_*(-,(Y/(Y-S))\wedge T)$,
\item[$\diamond$] $\Fr_*(-,X\times T^{n+1})=\Fr_*(-,(Y/(Y-S))\wedge T)$,
\item[$\diamond$] $\ZF^{qf}_*(-,X\times T^{n+1})=\ZF^{qf}_*(-,(Y/(Y-S))\wedge T)$,
\item[$\diamond$] $\ZF_*(-,X\times T^{n+1})=\ZF_*(-,(Y/(Y-S))\wedge T)$,
\end{itemize}
Finally, we have equalities of pointed sets:
\begin{itemize}
\item[$\diamond$] $\F_m(U,X\times T^{n+1})=\F_m(U,(Y/(Y-S))\wedge T)$,
\item[$\diamond$] $\F^{qf}_m(U,X\times T^{n+1})=\F^{qf}_m(U,(Y/(Y-S))\wedge T)$.
\end{itemize}
Specifying definitions of Section~\ref{sec:Fr_qf_and_ZF_qf}, a
section of $\Fr_m(-,(X\times T^{n+1}))$ on $U\in Sm/k$ is a tuple
   $$c=(Z,W,\phi_1,\ldots,\phi_{m};h:W\to X\times \A^n;f:W\to \A^1),$$
where $Z$ is a closed subset of $U\times\bb A^m$ finite over $U$,
$W$ is an \'{e}tale neighborhood of $Z$ in $\A^m\times U$,
$\phi_1,\ldots,\phi_{m},f$ are regular functions on $W$,
$h=(g,\phi_{m+1},\ldots,\phi_{m+n}):W\to X\times \A^n$ is a regular
map such that
   $$Z=Z(\phi_1,\ldots,\phi_{m},f)\cap h^{-1}(X\times \{0\})=Z(\phi_1,\ldots,\phi_{m},f,\phi_{m+1},\ldots,\phi_{m+n}).$$
The section $c$ is in $\Fr^{qf}_m(U,X\times T^{n+1})$ if and only
if the vanishing locus
$Z(\phi_1,\ldots,\phi_{m},\phi_{m+1},\ldots,\phi_{m+n})$ is
quasi-finite over $U$.

The section $c$ is in $\F_m(U,X\times T^{n+1})$ if and only if the
set $Z$ is connected. The section $c$ is in
$\F^{qf}_m(U,X\times T^{n+1})$ if and only if the set $Z$ is
connected and the vanishing locus
$Z(\phi_1,\ldots,\phi_{m},\phi_{m+1},\ldots,\phi_{m+n})$ is
quasi-finite over $U$.

The suspension map $\Sigma: \Fr_m(-,X\times T^{n+1}) \to
\Fr_{m+1}(-,X\times T^{n+1})$ sends
$(Z,W,\phi_1,\ldots,\phi_{m};g;\phi_{m+1},\ldots,\phi_{m+n};f)$ to
$(Z\times \{0\},W\times
\A^1,\phi_1,\ldots,\phi_{m},t;g;\phi_{m+1},\ldots,\phi_{m+n};f)$.

\begin{notation}{\rm
For the convenience of computations we shall write
$(Z,W,\phi_1,\ldots,\phi_{m},\phi_{m+1},\ldots,\phi_{m+n};f:W\to
\A^1,g:W\to X)$ for $(Z,W,\phi_1,\ldots,\phi_{m};g:W\to
X;\phi_{m+1},\ldots,\phi_{m+n};f:W\to \A^1)$ in the rest of the
paper.

}\end{notation}

Let $m\geq 0$ be an integer. Using Construction \ref{Pairs_to_Shv} the canonical morphism $(\A^1,\emptyset)\to(\A^1,\bb G_m)$
in $SmOp(\Fr_0(k)$ induces a morphism in $Shv_\bullet(Sm/k))$:
   $$p:\Fr_m(-,X\times T^{n}\times \A^1) \to \Fr_m(-,X\times T^{n+1})$$
sending
$(Z,W,\phi_1,\ldots,\phi_{m},\phi_{m+1},\ldots,\phi_{m+n};f;g)$ to
$(Z',W,\phi_1,\ldots,\phi_{m},\phi_{m+1},\ldots,\phi_{m+n},\phi_{m+n+1};g),$
where $Z'=Z\cap Z(f)$, $\phi_{m+n+1}=f$.

\begin{lemma}\label{l:pushout}
If $U$ is essentially $k$-smooth local henselian, then the image of
the map $p:\Fr_m(U,X\times T^{n}\times \A^1) \to
\Fr_m(U,X\times T^{n+1})$ is contained in $\Fr^{qf}_m(U,X\times
T^{n+1})$. Moreover, for any $m\geq 0$ the map $p$ sends
$\F_m(U,X\times T^{n}\times \A^1)$ to the set
$\F^{qf}_m(U,X\times T^{n+1})$. Finally, the square of pointed
sets
$$\xymatrix{\F_m(U,X\times T^{n}\times \Gm) \ar[r]^{i}\ar[d]_{}& \F_m(U,X\times T^{n}\times \A^1) \ar^{p}[d]\\
               {*} \ar[r]^{}&  \F^{qf}_m(U,X\times T^{n+1}), }$$
is a pushout square. Here ${*} \in \F_m(U,X\times T^{n+1})$ is the
empty framed correspondence.
\end{lemma}

\begin{proof}
The first assertion is obvious. To prove the second one, take an
element
   $$c=(Y,W,\phi_1,\ldots,\phi_{m+n},f:W\to \A^1;g:W\to X) \in \F_m(U,X\times T^{n}\times \A^1)$$
with $W=(\A^m\times U)^h_Y$. Then
$p(c)=(Z,W,\phi_1,\ldots,\phi_{m+n+1};g)$, where $Z=Y\cap Z(f)$,
$\phi_{m+n+1}=f$.

Since $Y$ is connected and finite over the henselian $U$, then $Y$
is henselian and local. Hence $W$ is local. By
Remark~\ref{r:henselian_properties} the closed subset $Z$ is
connected, whence the second assertion of the lemma.

To prove the third assertion, it sufficient to construct a section
   $$s:\F^{qf}_m(U,X\times T^{n+1})\setminus * \to \F_m(U,X\times T^{n}\times \A^1)\setminus \F_m(U,X\times T^{n}\times \Gm)$$
of $p$ and check that the map $p$ is injective on the complement of
$\F_m(U,X\times T^{n}\times \Gm)$.

We construct $s$ as follows. Take an element $c=(Z,\A^m\times
U\xleftarrow{can} W,\phi_1,\ldots,\phi_{m+n+1};g)$ in
$\F^{qf}_m(U,X\times T^{n+1})$ with a non-empty $Z$ and with
$W=(\A^m\times U)^h_Z$. Since $Z$ is connected and finite over the
local henselian $U$, the scheme $W$ is local. Set
   $$s(c)=(can(Y),\A^m\times U\xleftarrow{can} W,\phi_1,\ldots,\phi_{m+n};\phi_{m+n+1}:W\to \A^1;g:W\to X),$$
where $Y=Z(\phi_1,\ldots,\phi_{m+n})\subset W$. The set $Y$ is
quasi-finite over $U$, because $c\in \F^{qf}_m(U,X\times
T^{n+1})$. By Remark~\ref{r:henselian_properties} the set $Y$ is
finite over $U$ and connected and local. By
Lemma~\ref{l:closed-embedding} the morphism $can|_{Y}:Y\to
\A^m\times U$ is a closed embedding and $can^{-1}(can(Y))=Y$. Thus
$s(c)\in \F_m(U,X\times T^{n+1}\times \A^1).$ Clearly,
$p(s(c))=c$.

Now check the required injectivity for $p$. Let
$c'=(Y',W',\phi'_1,\ldots,\phi'_{m+n},f':W'\to \A^1;g':W'\to X)$ and
$c''=(Y'',W'',\phi''_1,\ldots,\phi''_{m+n},f'':W''\to
\A^1;g'':W''\to X)$ in $\F_m(U,X\times T^{n}\times \A^1)$ be two
elements with $W'=(\A^m\times U)^h_{Y'}$ and $W''=(\A^m\times
U)^h_{Y''}$. Let $can':W'\to \A^m\times U$ be the canonical morphism
and let $s':Y'\to W'$ be the section of $can'$.

Suppose that $p(c)=p(c')$ and the support $Z=Y'\cap Z(f')=Y''\cap
Z(f'')$ is non-empty. We must check that $c'=c''$. The element
$p(c')$ is of the form
   $$(Z,\A^m\times U\xleftarrow{can} W,\phi_1,\ldots,\phi_{m+n+1};g),$$
where $Z=Y'\cap Z(f')$, $W=(\A^m\times U)^h_Z$,
$\phi_{i}=\phi_i|_{W}$,$g=g|_W$. Consider the canonical morphism
$can_1: W\to W'$. It is the henselization of $W'$ at $s'(Y')$. Note
that $s'(Y')\subset W'$ is a closed subset containing $s'(Z)$.
Moreover, $Y'$ is finite over the henselian $U$ and connected. By
Lemma~\ref{l:lift} there is a unique section $t':Y'\to W$ of the
morphism $can_1$, and $t'(Y')=can^{-1}_1(Y')$ contains $s(Z)$, where
$s:Z\to W$ is the section of $can_1$ (the morphism $s$ is also the
section of $can=can'\circ can_1$). These arguments imply an equality
   $$c'=(Y',W,\phi'_1|_W,\ldots,\phi'_{m+n}|_W,f'|_W:W\to \A^1;g'|_W:W\to X) \in \F_m(U,X\times T^{n}\times \A^1).$$
By the same reason one has an equality
   $$c''=(Y'',W,\phi''_1|_W,\ldots,\phi''_{m+n}|_W,f''|_W:W\to \A^1;g''|_W:W\to X) \in \F_m(U,X\times T^{n}\times \A^1).$$
Since $W$ is the henselization of $\A^m\times U$ at $Z$ and
$p(c')=p(c'')$, one has equalities $\phi'_i|_W=\phi''_i|_W$ for
$i=1,...,m+n$, $f'|_W=f''_W$ and $g'|_W=g''|_W$. Hence $Y'=Y''$ and,
moreover, $c'=c''$ in $\F_m(U,X\times T^{n}\times \A^1)$. The
desired injectivity is proved. The section $s$ is constructed above
and our lemma follows.
\end{proof}

\begin{corollary}
\label{cor:pushout}
For any integer $n\geq 0$ the natural morphism
$$\alpha_*: \ZF_*(X\times T^n \times \A^1)/\ZF_*(X\times T^n \times \Gm)\to \ZF^{qf}_*(X\times T^{n+1})$$
of $\ZF_*(k)$-presheaves is an isomorphism locally for the Nisnevich
topology. As a consequence, the natural morphism
   $$\alpha: \ZF(X\times T^n \times \A^1)/\ZF(X\times T^n \times \Gm)\to \ZF^{qf}(X\times T^{n+1})$$
of $\ZF_*(k)$-presheaves is an isomorphism locally for the Nisnevich topology.
\end{corollary}

We are now in a position to prove Theorem~\ref{p:main}.

\begin{proof}[Proof of Theorem~\ref{p:main}]
The theorem is implied by Definitions~\ref{def:ZF_qf},
\ref{def:F_m_and_F_m_qf} and Corollary~\ref{cor:pushout}.
\end{proof}

We refer the reader to~\cite{GP2} for the definition of quasi-stability of framed presheaves.

\begin{lemma}\label{l:weak-eq_and_C_*}
Let $k$ be an infinite perfect field. Let $A$ and $B$ be linear
framed presheaves such that the cohomology presheaves of the
complexes $C_*(A)$ and $C_*(B)$ are quasi-stable.
Let $\alpha: A\to B$ be a morphism of linear framed presheaves,
which is an isomorphism locally in the Nisnevich topology. Then the morphism
   $$C_*(\alpha): C_*(A)\to C_*(B)$$
is a quasi-isomorphism locally in the Nisnevich topology.
\end{lemma}

\begin{proof}
By assumption the map of the Eilenberg--Mac~Lane spectra
$\alpha:EM(A)\to EM(B)$ is a local weak equivalence. Hence the
induced map $\alpha:EM(C_*(A))\to EM(C_*(B))$ is a motivic weak
equivalence of $S^1$-spectra. By assumption, the presheaves of stable homotopy
groups of the spectra $EM(C_*(A)), EM(C_*(B))$ are radditive,
quasi-stable and $\bb A^1$-invariant (see~\cite{Voeradd}
for the definition of radditivity).

Let $EM(C_*(A))_f,EM(C_*(B))_f$ be fibrant replacements of
$EM(C_*(A)),EM(C_*(B))$ in the level injective model structure of
$S^1$-spectra. Then they are motivically fibrant spectra
by~\cite[7.4]{GP1}. Hence the stable motivic equivalence of
$S^1$-spectra $\alpha_f:EM(C_*(A))_f\to EM(C_*(B))_f$ is a Nisnevich
local weak equivalence of $S^1$-spectra. Thus the morphism of
complexes $C_*(\alpha): C_*(A)\to C_*(B)$ is a quasi-isomorphism
locally in the Nisnevich topology.
\end{proof}

We are now in a position to prove a statement which is necessary for
the proof of Theorem~\ref{th:Main}.

\begin{proposition}\label{p:factor_complex}
Let $k$ be an infinite perfect field. Then the morphism
   $$C_*(\alpha): C_*\ZF(X\times T^n \times \A^1)/C_*\ZF(X\times T^n \times \Gm)
       \to C_*\ZF^{qf}(X\times T^{n+1})$$
is a quasi-isomorphism locally in the Nisnevich topology.
\end{proposition}

\begin{proof}
The morphism $\alpha: \ZF(X\times T^n \times \A^1)/\ZF(X\times
T^n \times \Gm)\to \ZF^{qf}(X\times T^{n+1})$ is a morphism of
linear framed presheaves. Set $A=\ZF(X\times T^n \times \A^1)/\ZF(X\times T^n \times \Gm)$ and $B=\ZF^{qf}(X\times
T^{n+1})$. Then the cohomology presheaves of the complexes $C_*(A)$
and $C_*(B)$ are quasi-stable by the construction of $A$ and $B$.
Now Theorem~\ref{p:main} and Lemma~\ref{l:weak-eq_and_C_*} imply the
claim.
\end{proof}

Thus we have computed the complex $C_*\ZF(X\times T^n
\times \A^1)/C_*\ZF(X\times T^n \times \Gm)$  locally in the Nisnevich topology
as the complex
$C_*\ZF^{qf}(X\times T^{n+1})$. Our next goal is to show that the
latter complex is quasi-isomorphic to $C_*\ZF(X\times T^{n+1})$
locally in the Zariski topology verifying Theorem~\ref{p:moving}.
The next two sections are dedicated to this theorem.

We finish this section with the following remark justifying the use of property of
quasi-stability of linear framed presheaves.

\begin{remark}\label{justification}{\rm
Lemma \ref{l:weak-eq_and_C_*} is also true if the condition
``the cohomology presheaves of the
complexes $C_*(A)$ and $C_*(B)$ are quasi-stable"
is replaced by the condition ``the cohomology presheaves of the
complexes $C_*(A)$ and $C_*(B)$ are stable":
the morphism
   $C_*(\alpha): C_*(A)\to C_*(B)$
is a quasi-isomorphism locally in the Nisnevich topology.

However, we can not apply this for the proof of Proposition \ref{p:factor_complex}.
Indeed,
the cohomology presheaves of the complex
$$C_*\ZF(X\times T^n\times \A^1)/C_*\ZF(X\times T^n \times \Gm)$$
are quasi-stable only. And it is not readily apparent that they are stable.



%
}
\end{remark}

\section{A filtration on $\ZF_n(-,X\times T^{n+1})$}

We start with the following

\begin{definition}\label{d:defining_set}{\rm
Let $U\in Sm/k$ be an affine variety and let $c=(Z,W,\phi;g)\in
\Fr_m(U,X\times T^{n+1})$ be a framed correspondence. A finite set
of polynomials $F_1,\ldots, F_r \in k[\A^{m+n+1}\times U]$ is said
to be {\it $c$-defining\/} if for every point $u\in U$ there is
$i\in \{1,2...,r\}$ such that
\begin{itemize}
\item[$\diamond$] the polynomial $F_i(-,u) \in k(u)[\A^{m+n+1}]$ is
nonzero,
\item[$\diamond$] $\phi(W_u)\subseteq Z(F_i(-,u))$ in
$\A^{m+n+1}_u$.
\end{itemize}
Note that if a finite set of polynomials $F_1,\ldots, F_r \in
k[\A^{m+n+1}\times U]$ is $c$-defining, where
$c=(Z,W,\phi;g)=(Z_1\sqcup Z_2,W,\phi;g)$, then the same collection
of polynomials is $(Z_1,W-Z_2,\phi;g)$-defining and is
$(Z_2,W-Z_1,\phi;g)$-defining respectively.

}\end{definition}

\begin{remark}{\rm
We want to explain how we use a finite $c$-defining set of polynomials
$F_1,\ldots, F_r \in k[\A^{m+n+1}\times U]$.
Take an integer $d$ strictly greater than the
degrees of all $F_i$-s. Lemma~\ref{l:homotopy_properties} below gives rise to a
``homotopy"
$$h^d_s(c)\in \Fr_n(U\times\bb A^1,X\times T^{n+1})$$
between
$h^d_0(c)= c\in \Fr_n(U,X\times T^{n+1})$
and
$h^d_1(c)= t_d(c)\in \Fr^{qf}_n(U,X\times T^{n+1})$.
Moreover, if $c \in \Fr^{qf}_m(U,X\times T^{n+1})$,
then $h^d_s(c) \in \Fr^{qf}_m(U\times \A^1,X\times T^{n+1})$.
Finally, if $Z$ is the support of $c$, then $Z\times \bb A^1$ is the support of $h^d_s(c)$.
}\end{remark}

The following lemma is crucial in our analysis.

\begin{lemma}
Let $m,n\geq 0$ and $Y$ an affine (possibly non-irreducible)
$k$-variety. Let $W\to \A^m\times Y$ be an \'{e}tale morphism and
$\psi: W\to \A^{m+n+1}\times Y$ be a morphism of $Y$-schemes. Then
there is a finite set of polynomials $F_1,\ldots, F_r \in
k[\A^{m+n+1}\times Y]$ such that for every point $y\in Y$ there is
$i\in \{1,2...,r\}$ such that
\begin{itemize}
\item[$\diamond$] the polynomial $F_i(-,y) \in k(y)[\A^{m+n+1}]$ is nonzero and
\item[$\diamond$] $\psi(W_y)\subseteq Z(F_i(-,u))$ in $\A^{m+n+1}_u$.
\end{itemize}
\end{lemma}

\begin{proof}
We proceed by induction in the dimension of $Y$. If $dim(Y)=0$ then
there is nothing to prove. Now suppose $dim(Y)>0$. Let $Y_1,...,Y_l$
be all irreducible components of $Y$. For an index $i$ from
$\{1,...,l\}$ take the restriction of the map $\psi|_{W_{Y_i}}\colon
W_{Y_i}\to\A^{m+n+1}\times Y_i$ to $Y_i$. The dimension of $W_{Y_i}$
is $m+dim(Y_i)$. Thus the closure of its image is contained in the
zero locus $Z(\bar F_{(i)})$ of a non-zero polynomial $\bar
F_{(i)}\in k[\A^{m+n+1}\times Y_i]$. Since $Y_i$ is closed in the
affine variety $Y$, the polynomial $\bar F_i$ can be extended to a
polynomial $F_i\in k[\A^{m+n+1}\times Y]$. Let $V\subset Y$ be an
open subset consisting of those $y\in Y$ such that there is $i$ with
$0\neq F_i(-,y)$ in $k(y)[\A^{m+n+1}]$. Let $Y'=Y-V$.

By construction, $V$ has a non-empty intersection with every
irreducible component of $Y$. Thus $dim(Y')<dim(Y)$. By the
inductive assumption there are polynomials $\bar F_{l+1},...,\bar
F_{r}$ in $k[\A^{m+n+1}\times Y']$ such that for every point $y\in
Y'$ there is $j\in \{l+1,...,r\}$ with each polynomial $\bar
F_j(-,y) \in k(y)[\A^{m+n+1}]$ nonzero and $\psi(W_y)\subseteq
Z(\bar F_j(-,y))$ in $\A^{m+n+1}_y$. Since $Y'$ is closed in the
affine $Y$ for any $j$, the polynomial $\bar F_j$ can be extended to
a polynomial $F_j$ in $k[\A^{m+n+1}\times Y]$. Clearly, the set of
polynomials $F_i$, where $i\in \{1,2,...,r\}$, are the desired
polynomials for $Y$.
\end{proof}

The preceding lemma has the following

\begin{corollary}\label{cor:existdef}
Let $U\in Sm/k$ be an affine variety and let $c=(Z,W,\phi;g)\in
\Fr_m(U,X\times T^{n+1})$ be a framed correspondence. Then there
exists a $c$-defining set of polynomials $F_1,\ldots, F_r \in
k[\A^{m+n+1}\times U]$.

Moreover, if $f: V\to U$ is a morphism of
$k$-smooth affine varieties and $F_1,\ldots, F_r \in
k[\A^{m+n+1}\times U]$ is a $c$-defining set, then $f^*(F_1),\ldots,
f^*(F_r) \in k[\A^{m+n+1}\times V]$ is a $f^*(c)$-defining set.
\end{corollary}

Let $U\in Sm/k$ be an affine variety. Let $d>0$. Define
$\Fr_m^{<d}(U,X\times T^{n+1})$ as a subset of $\Fr_m(U,X\times
T^{n+1})$ consisting of those $c=(Z,W,\phi;g)\in \Fr_m(U,X\times
T^{n+1})$ for which there exists a $c$-defining set $F_1,\ldots F_r
\in k[\A^{m+n+1}\times U]$ with $\deg F_i<d$ for all $i=1,\ldots,r$.
Set,
   $$(\Fr^{qf}_m)^{<d}(U,X\times T^{n+1}):= \Fr^{qf}_m(U,X\times T^{n+1}) \cap\Fr_m^{<d}(U,X\times T^{n+1}).$$

Corollary~\ref{cor:existdef}
shows that each of these filtrations is stable under pullbacks.
Hence one has the following obvious

\begin{lemma}\label{l:exhauting_1}
For any integers $m,n\geq 0$ and any integer $d>0$ the following
statements are true:
\begin{itemize}
\item[(i)]
$\Fr_m^{<d}(-,X\times T^{n+1})$ is a subpresheaf on $AffSm/k$ of
the presheaf $\Fr_m(-,X\times T^{n+1})$;
\item[(ii)]
the increasing filtration of the presheaf $\Fr_m(-,X\times
T^{n+1})|_{AffSm/k}$ by subpresheaves $\Fr^{<d}_m(-,X\times
T^{n+1})$ is exhausting;
\item[(iii)]
$(\Fr^{qf}_m)^{<d}(-,X\times T^{n+1})$ is a subpresheaf on
$AffSm/k$ of the presheaf $\Fr^{qf}_m(-,X\times T^{n+1})$;
\item[(iv)]
the increasing filtration of the presheaf $\Fr^{qf}_m(-,X\times
T^{n+1})|_{AffSm/k}$ by subpresheaves $(\Fr^{qf}_m)^{<d}(-,X\times
T^{n+1})$ is exhausting.
\end{itemize}
\end{lemma}

\begin{definition}\label{d:filtration_on_ZF_qf_m}{\rm
For an affine $k$-smooth $U$ we define $\ZF_m^{<d}(U,X\times T^{n+1})$ as
$$\mathbb Z[\Fr_m^{<d}(U,X\times T^{n+1})]/
\langle(Z_1\sqcup Z_2,W,\phi;g)-(Z_1,W_2,\phi|_{W_2};g|_{W_2})
-(Z_2,W_1,\phi|_{W_1};g|_{W_1})\rangle,$$ where $W_i=W-Z_i$ for
$i=1,2$. By Lemma~\ref{l:exhauting_1} the assignment $U\mapsto
\ZF_m^{<d}(U,X\times T^{n+1})$ is a presheaf on $AffSm/k$.

Likewise, for an affine $k$-smooth $U$ define
$(\ZF^{qf}_m)^{<d}(U,X\times T^{n+1})$ as
$$\mathbb Z[(\Fr^{qf}_m)^{<d}(U,X\times T^{n+1})]/
\langle(Z_1\sqcup
Z_2,W,\phi;g)-(Z_1,W_2,\phi|_{W_2};g|_{W_2})-(Z_2,W_1,\phi|_{W_1};g|_{W_1})\rangle,$$
where $W_i=W-Z_i$ for $i=1,2.$ By Lemma~\ref{l:exhauting_1} the
assignment $U\mapsto (\ZF^{qf}_m)^{<d}(U,X\times T^{n+1})$ is a
presheaf on $AffSm/k$.

Recall that $\ZF_m(-,X\times T^{n+1})$ is a presheaf on $Sm/k$
(see Definition~\ref{stab}) given by
   $$\mathbb Z[\Fr_m(U,X\times T^{n+1})]/\langle(Z_1\sqcup Z_2,W,\phi;g)-(Z_1,W_2,\phi|_{W_2};g|_{W_2})-(Z_2,W_1,\phi|_{W_1};g|_{W_1})\rangle,$$
where $W_i=W-Z_i$ for $i=1,2$. In turn, $\ZF^{qf}_m(-,X\times
T^{n+1})$ is a presheaf on $Sm/k$ (see Definition~\ref{def:ZF_m_qf})
given by
   $$\mathbb Z[\Fr^{qf}_m(U,X\times T^{n+1})]/\langle(Z_1\sqcup Z_2,W,\phi;g)-(Z_1,W_2,\phi|_{W_2};g|_{W_2})-(Z_2,W_1,\phi|_{W_1};g|_{W_1})\rangle,$$
where $W_i=W-Z_i$ for $i=1,2.$

}\end{definition}

For any positive integers $d < d'$ the inclusion
$\Fr_m^{<d}(-,X\times T^{n+1})\subset \Fr_m^{<d}(-,X\times
T^{n+1})$ induces a morphism of
presheaves of Abelian groups on $AffSm/k$
   $$\ZF_m^{<d}(U,X\times T^{n+1})\to \ZF_m^{<d'}(U,X\times T^{n+1}).$$
Likewise, for any positive $d < d'$ the inclusion
$(\Fr^{qf}_m)^{<d}(-,X\times T^{n+1})\subset
(\Fr^{qf}_m)^{<d'}(-,X\times T^{n+1})$ induces a morphism of
presheaves of Abelian groups on $AffSm/k$
   $$(\ZF^{qf}_m)^{<d}(U,X\times T^{n+1})\to (\ZF^{qf}_m)^{<d'}(U,X\times T^{n+1}).$$

The next statement is a consequence of Lemma~\ref{l:exhauting_1}.

\begin{corollary}\label{cor:exhauting_2}
One has two equalities of presheaves on $AffSm/k$
   $$\colim_{d} \ZF_m^{<d}(-,X\times T^{n+1})=\ZF_m(-,X\times T^{n+1})$$
and
   $$\colim_{d} (\ZF^{qf}_m)^{<d}(-,X\times T^{n+1})=\ZF^{qf}_m(-,X\times T^{n+1}).$$
\end{corollary}

\begin{proof}
This is straightforward.
\end{proof}

\section{Moving lemma}\label{mlemma}

\begin{lemma}\label{lem:tnonzero}\footnote{We thank A.~Ananyevskiy for suggesting Lemma~\ref{lem:tnonzero}
in its present form.} Let $L$ be a field and $F\in L[x_1,\ldots
x_{r+1}]$ a nonzero polynomial such that $\deg F\leqslant (d-1)$.
Then the polynomials $F(t^{d^r},t^{d^{r-1}},\ldots, t^d,t)$ and
$F(t^d,\ldots,t^{d^{r-1}},t^{d^r},t)$ are both non-zero in $L[t]$.
Moreover, for any non-zero $s \in L$ the polynomials
$F(st^{d^r},st^{d^{r-1}},\ldots, st^d,t)$ and
$F(st^d,\ldots,st^{d^{r-1}},st^{d^r},t)$ are both non-zero.
\end{lemma}

\begin{proof}
Let us prove that the first polynomial is non-zero. If
$F=\sum_{I=(i_1,\ldots, i_{r+1})} a_I x^I$ then
   \begin{equation}\label{zvezda}
    F(t^{d^r},t^{d^{r-1}},\ldots, t^r,t)=\sum a_I t^{i_1d^r+\ldots+ i_nr+\ldots +i_{n+1}}
   \end{equation}
Let us check that if $I$ and $J$ are two different multi-indices,
then $I\cdot (d^r,\ldots,d,1)\neq J\cdot (d^r,\ldots,d,1)$. Indeed,
if these are equal, then
   \[d^r(i_1-j_1)+\cdots+d(i_n-j_n)+(i_{n+1}-j_{n+1})=0.\]
It follows that $i_{n+1}-j_{n+1}$ is divisible by $d$, but
$|i_{n+1}-j_{n+1}|\leqslant d-1$, hence $i_{n+1}=j_{n+1}$. Then
$(i_n-j_n)$ is divisible by $d$, hence zero and $(i_1-j_1)$ is zero
by induction. Thus all powers of $t$ in the sum~\eqref{zvezda} are
distinct. So if some $a_I\neq 0$ then the right hand side
of~\eqref{zvezda} is nonzero. The second polynomial is obtained from
the first one by permuting powers of $t$. Thus it is non-zero. For
any multi-index $I$ the coefficient at $t^{i_1d^r+\ldots+
i_nr+\ldots +i_{n+1}}$ in the polynomial
$F(st^{d^r},st^{d^{r-1}},\ldots, st^d,t)$ is obtained from $a_I$ by
multiplying a power of $s$. We see that the polynomial
$F(st^{d^r},st^{d^{r-1}},\ldots, st^d,t)$ is non-zero, and hence so
is $F(st^d,\ldots,st^{d^{r-1}},st^{d^r},t)$.
\end{proof}

Let $U\in Sm/k$ be an affine variety and let $c=(Z,W,\phi)\in
\Fr_m(U,X\times T^{n+1};g)$ be a framed correspondence. Let $d>0$
be an integer. Set
   $$t_d(c)=(Z,W,\phi_1-\phi_{m+n+1}^d,\phi_2-\phi_{m+n+1}^{d^2},\ldots,\phi_{m+n}-\phi_{m+n+1}^{d^{m+n}},\phi_{m+n+1};g).$$
Note that
$Z(\phi_1-\phi_{m+n+1}^d,\phi_2-\phi_{m+n+1}^{d^2},\ldots,\phi_{m+n}-\phi_{m+n+1}^{d^{m+n}},\phi_{m+n+1})=Z(\phi_1,\ldots,\phi_{m+n},\phi_{m+n+1})$
in $W$. Thus the tuple $t_d(c)$ is an element in $\Fr_n(U,X\times
T^{n+1})$ such that its support $Z$ is the same with that of $c$. If
$\tau$ is a variable, we put
   $$h^d(c)=(Z\times \A^1,W\times \A^1,\phi_1-\tau\phi_{m+n+1}^d,\phi_2-\tau\phi_{m+n+1}^{d^2},\ldots,\phi_{m+n}-\tau\phi_{m+n+1}^{d^{m+n}},\phi_{m+n+1};g).$$
Note that
$Z(\phi_1-\tau\phi_{m+n+1}^d,\ldots,\phi_{m+n}-\tau\phi_{m+n+1}^{d^{m+n}},\phi_{m+n+1})=Z\times
\A^1$ in $W\times\bb A^1$. Thus the tuple $h^d(c)$ is an element in
$\Fr_n(U\times\bb A^1,X\times T^{n+1})$ whose support equals
$Z\times \A^1$.

\begin{remark}\label{r:quasi-finite}{\rm
Let $c=(Z,W,\phi;g)\in \Fr_m(U,X\times T^{n+1})$ and let $\pi:
W\to U$ be the composite map $U\xleftarrow{pr_U} U\times
\A^{m}\leftarrow W$. Then the map $(\phi,\pi)\colon
W\to\A^{m+n+1}\times U$ is quasi-finite over $0\times U$. Hence it
is quasi-finite over some Zariski open neighborhood $V$ of $0\times
U$. Then $W'=(\phi,\pi)^{-1}(V)$ is a Zariski open neighborhood of
$Z$ in $W$. Replacing $W$ by $W'$ if necessary, we may and shall
always assume in what follows that $(\phi,\pi): W\to
\A^{m+n+1}\times U$ is quasi-finite.

}\end{remark}

\begin{lemma}\label{l:homotopy_properties}
If $U\in AffSm/k$ then the following statements are true:
\begin{itemize}
\item [(i)] if $c \in \Fr^{<d}_m(U,X\times T^{n+1})$, then $t_d(c) \in \Fr^{qf}_m(U,X\times T^{n+1})$;
\item [(ii)] if $c \in (\Fr^{qf}_m)^{<d}(U,X\times T^{n+1})$, then $h^d(c) \in \Fr^{qf}_m(U\times \A^1,X\times
T^{n+1})$.
\end{itemize}
\end{lemma}

\begin{proof}
Prove the first assertion. Let $c=(Z,W,\phi;g)\in \Fr_m(U,X\times
T^{n+1}) \in \Fr^{<d}_m(U,X\times T^{n+1})$. Let $F_1,\ldots F_r
\in k[\A^{m+n+1}_U]$  be a $c$-defining set with $\deg F_i<d$ for
all $i=1,\ldots,r$. We must check that $t_d(c)$ is in
$\Fr^{qf}_m(U,X\times T^{n+1})$. So, take
$$Y=Z(\phi_1-\phi_{m+n+1}^d,\phi_2-\phi_{m+n+1}^{d^2},\ldots,\phi_{m+n}-\phi_{m+n+1}^{d^{m+n}})\subset W.$$
Let $\pi:W\to U$ be the composite map $U\xleftarrow{pr_U} U\times
\A^{m}\leftarrow W$ as in Remark \ref{r:quasi-finite}.
We must check that $\pi|_Y: Y\to U$ is quasi-finite.
So we must check that for any point $u\in U$ the fiber $Y(u)$ of $Y$
over $u$ is finite. Let $\theta: \A^1\to \A^{m+n+1}$ be a morphism
taking a point $t$ to the point $(t^d,t^{d^2},...,t^{d^{m+n}},t)$.
It is a closed embedding with the image $C=\theta(\A^1)$.

By Remark \ref{r:quasi-finite} the morphism $\psi=(\phi,\pi):W\to
\A^{m+n+1}\times U$ is quasi-finite. Given a point $u\in U$ there is
a polynomial $F$ from the $c$-defining set such that $F(-,u)$ is
non-zero and its vanishing locus $Z(F(-,u))$ in $\A^{m+n+1}_u$
contains $\psi(W(u))$. Clearly, $Y(u)$ is contained in the set
   $$\psi^{-1}(Z(F(-,u))\cap C).$$
The set $Z(F(-,u))\cap C$ is in a bijection with the vanishing locus
of the polynomial $F(t^d,t^{d^2},...,t^{d^{m+n}},t)$ on the line
$\A^1$ with the coordinate $t$. Thus by Lemma~\ref{lem:tnonzero} the
set $Z(F(-,u))\cap C$ is finite, and hence so is $Y(u)$. The first
assertion is proved.

Let us verify the second assertion. Let $c=(Z,W,\phi;g) \in
(\Fr^{qf}_m)^{<d}(U,X\times T^{n+1})$ and let $F_1,\ldots F_r \in
k[\A^{m+n+1}_U]$  be a $c$-defining set with $\deg F_i<d$ for all
$i=1,\ldots,r$. We must check that $h^d(c)$ is in
$\Fr^{qf}_m(U\times \A^1,X\times T^{n+1})$. So, take
   $$Y_{\tau}=Z(\phi_1-\tau\phi_{m+n+1}^d,\phi_2-\tau\phi_{m+n+1}^{d^2},\ldots,\phi_{m+n}-\tau\phi_{m+n+1}^{d^{m+n}})\subset W\times \A^1,$$
where $\tau$ is the coordinate on the additional factor $\A^1$.
Let $\pi: W\to U$ be as above in this proof. Consider the map
$\Pi=\pi\times id_{\A^1}: W\times \A^1\to U\times \A^1$.
We must check that
$\Pi|_{Y_{\tau}}: Y_{\tau}\to U\times \A^1$
is quasi-finite. So we must
check that for any point $v\in U\times \A^1$ the fiber $Y_{\tau}(v)$ of
$Y_{\tau}$ over $v$ is finite. Replacing the base field $k$ by its
algebraic closure $\bar k$, we may assume that any point $v\in
U\times \A^1$ is of the form $(u,a)$ with $a\in \bar k$. So, we must
check that for any point $(u,a)\in U\times \A^1$ the fiber $Y_a(u)$
of $Y_{\tau}$ over $(u,a)$ is finite.

Given $0\neq a\in \bar k$, let $\theta_a: \A^1\to \A^{m+n+1}$ be a
morphism taking a point $t$ to the point
$(at^d,at^{d^2},...,at^{d^{m+n}},t)$. It is a closed embedding with
the image $C_a=\theta_a(\A^1)$. For $a=0$ let $\theta_0: \A^1\to
\A^{m+n+1}$ be the morphism taking a point $t$ to $(0,0,...,0,t)$.
It is a closed embedding with the image $C_0=\theta_0(\A^1)$. It is
the last coordinate line $\A^1_{m+n+1}$ in $\A^{m+n+1}$.

By Remark~\ref{r:quasi-finite} the morphism $\psi=(\phi,\pi):W\to
\A^{m+n+1}\times U$ is quasi-finite. Given a point $u\in U$, there
is a polynomial $F$ from the $c$-defining set such that $F(-,u)$ is
non-zero and its vanishing locus $Z(F(-,u))$ in $\A^{m+n+1}_u$
contains $\psi(W(u))$. For a given $0\neq a\in \bar k$ the set
$Y_a(u)$ is contained in the set
   $$\psi^{-1}(Z(F(-,u))\cap C_a).$$
The set $Z(F(-,u))\cap C_a)$ is in a bijection with the vanishing
locus of the polynomial $F(at^d,at^{d^2},...,at^{d^{m+n}},t)$ on the
line $\A^1$ with the coordinate $t$. Thus by
Lemma~\ref{lem:tnonzero} the set $Z(F(-,u))\cap C_a)$ is finite in
this case, and hence so is $Y_a(u)$.

For $a=0$, the set $Y_0(u)$ coincides with the closed subset
$Z(\phi_1,...,\phi_{m+n})$ in $W$. It is quasi-finite over $U$,
because $c=(Z,W,\phi;g) \in (\Fr^{qf}_m)^{<d}(U,X\times T^{n+1})$.
The second assertion is proved.
\end{proof}

Lemma~\ref{l:homotopy_properties} implies that the assignment
$c\mapsto t_d(c)$ gives a morphism of presheaves of pointed sets
   $$t_d: \Fr^{<d}_m(-,X\times T^{n+1}) \to \Fr^{qf}_m(-,X\times T^{n+1})$$
on $AffSm/k$. It also implies that the assignment $c\mapsto h^d(c)$
gives a morphism of presheaves of pointed sets
   $$h_d: \Fr^{<d}_m(-,X\times T^{n+1}) \to \Fr^{qf}_m(-\times \A^1,X\times T^{n+1})$$
on $AffSm/k$. Finally, the assignment $c\mapsto h_d(c)$ gives a
morphism of presheaves of pointed sets
   $$h_{d}^{qf}:(\Fr^{qf}_m)^{<d}(-,X\times T^{n+1}) \to \Fr^{qf}_m(-\times \A^1,X\times T^{n+1})$$
on $AffSm/k$.

Consider a diagram
$$\xymatrix{\Fr^{<d}_m(-,X\times T^{n+1}) \ar[rrr]^(.5){i_d}  \ar[rrrd]^{t_d} &&& \Fr_m(-,X\times T^{n+1}) \\
              (\Fr^{qf}_m)^{<d}(-,X\times T^{n+1}) \ar[u]^{in_d}  \ar[rrr]^(.5){j_d} &&& \Fr^{qf}_m(-,X\times T^{n+1}) \ar[u]^{in} }$$
of presheaves of pointed sets on $AffSm/k$.
Lemma~\ref{l:homotopy_properties} shows that $h_d:
\Fr^{<d}_m(-,X\times T^{n+1}) \to
\underline{\Hom}(\A^1,\Fr_m(-,X\times T^{n+1}))$ is an
$\A^1$-homotopy between the morphisms $in\circ t_d$ and $i_d$. It
also shows that $h_{d}^{qf}: (\Fr^{qf}_m)^{<d}(-,X\times T^{n+1})
\to \underline{\Hom}(\A^1,\Fr^{qf}_m(-,X\times T^{n+1}))$ is an
$\A^1$-homotopy between the morphisms $t_d\circ in_d$ and $j_d$.

Applying the free abelian group functor to the morphisms
$t_d$,$i_d$,$j_d$,$in_d$ and $in$, we get certain morphisms between
presheaves of abelian groups as well as two $\A^1$-homotopies
(namely, $\mathbb Z [t_d]$, $\mathbb Z [i_d]$, $\mathbb Z [j_d]$,
$\mathbb Z [in_d]$, $\mathbb Z [in]$, $\mathbb Z [h_d]$ and $\mathbb
Z [h_d^{qf}]$). Note that these morphisms and these two homotopies
respect the additivity relations. Thus following
Definition~\ref{d:filtration_on_ZF_qf_m}, we finally get morphisms
$I_d$,$J_d$,$In_d$,$In$ and a morphism of presheaves
   $$T_d: \ZF^{<d}_m(-,X\times T^{n+1}) \to \ZF^{qf}_m(-,X\times T^{n+1}),$$
and two $\A^1$-homotopies $H_d$, $H_d^{qf}$. In this way we get a
diagram
   $$\xymatrix{\ZF^{<d}_m(-,X\times T^{n+1}) \ar[rrr]^(.4){I_d}  \ar[rrrd]^{T_d} &&& \ZF_m(-,X\times T^{n+1}) \\
              (\ZF^{qf}_m)^{<d}(-,X\times T^{n+1}) \ar[u]^{In_d}  \ar[rrr]^(.4){J_d} &&& \ZF^{qf}_m(-,X\times T^{n+1}) \ar[u]^{In} }$$
of presheaves of abelian groups on $AffSm/k$.

We document these arguments as follows.

\begin{lemma}\label{l:two_homotopies}
The $\A^1$-homotopy $h_d$ yields an $\A^1$-homotopy
$$H_d: \ZF^{<d}_m(-,X\times T^{n+1}) \to \underline{\Hom}(\A^1,\ZF_m(-,X\times T^{n+1}))$$
between $In\circ T_d$ and $I_d$. The $\A^1$-homotopy $h_{d}^{qf}$ yields an $\A^1$-homotopy
   $$H_{d}^{qf}: (\ZF^{qf}_m)^{<d}(-,X\times T^{n+1}) \to \underline{\Hom}(\A^1,\ZF^{qf}_m(-,X\times T^{n+1}))$$
between $T_d\circ In_d$ and $J_d$.
\end{lemma}

\begin{proposition}\label{ZF_qf_and_ZF}
For any integers $m,n\geq 0$ the morphism
$$In: C_*\ZF^{qf}_m(-,X\times T^{n+1}) \to C_*\ZF_m(-,X\times T^{n+1})$$
of complexes of presheaves of Abelian groups is a section-wise quasi-isomorphism
on the category $AffSm/k$.
\end{proposition}

\begin{proof}[Proof of Proposition \ref{ZF_qf_and_ZF}]
The functor $C_*$ converts $\A^1$-homotopies to naive simplicial homotopies.
Now the proposition follows from Lemma \ref{l:two_homotopies}
and Corollary~\ref{cor:exhauting_2}.

In more detail, for each integer $r\geq 0$ and each $k$-scheme $U\in AffSm/k$,
write $\mathbb H_r(U)$ for the $r$th homology group of the complex
$C_*\ZF_m(U,X\times T^{n+1})$ and write $\mathbb H^{qf}_r(U)$ for the $r$th homology group
of the complex $C_*\ZF^{qf}_m(U,X\times T^{n+1})$.
In addition, if $d>0$ is an integer, then write
$\mathbb H^{(<d)}_r(U)$ for the $r$th homology group of $C_*\ZF^{<d}_m(U,X\times T^{n+1})$
and $(\mathbb H^{qf})^{(<d)}_r(U)$ for the $r$th homology group of $C_*(\ZF^{qf}_m)^{<d}(U,X\times T^{n+1})$.
We choose the symbol $\mathbb H$ to avoid any confusion with
$\A^1$-homotopies from Lemma \ref{l:two_homotopies}. First, we have a commutative diagram of groups
   \begin{equation}\label{aaa}
    \xymatrix{\colim_{d} \mathbb H^{(<d)}_r(U) \ar[rrr]^(.5){I}  &&& \mathbb H_r(U) \\
              \colim_{d} (\mathbb H^{qf})^{(<d)}_r(U) \ar[u]^{\colim_{d} (In_d)_*}  \ar[rrr]^(.5){J} &&& \mathbb H^{qf}_r(U) \ar[u]^{In_*} }
    \end{equation}
in which the arrows $I$ and $J$ are isomorphisms by Lemma \ref{l:exhauting_1}. Second, for each integer $d>0$ we have
a commutative diagram of groups
    \begin{equation}\label{bbb}
       \xymatrix{\mathbb H^{(<d)}_r(U) \ar[rrr]^(.5){(I_d)_*} \ar[rrrd]^{(T_d)_*}  &&& \mathbb H_r(U) \\
              (\mathbb H^{qf})^{(<d)}_r(U) \ar[u]^{(In_d)_*}  \ar[rrr]^(.5){(J_d)_*} &&& \mathbb H^{qf}_r(U) \ar[u]^{In_*} }
    \end{equation}
Since the arrows $I$ and $J$ from the commutative diagram~\eqref{aaa} are isomorphisms and the
upper (respectively, lower) triangle of the diagram~\eqref{bbb} is commutative, it follows that $In_*$ is surjective (respectively, injective).
We see that $In_*$ is an isomorphism. The proposition is proved.
\end{proof}

We are now in a position to prove Theorem~\ref{p:moving}

\begin{proof}[Proof of Theorem~\ref{p:moving}]
By Definitions~\ref{def:ZF_qf} and \ref{stab} the complexes
$C_*\ZF^{qf}(-,X\times T^{n+1}))$ and $C_*\ZF(-,X\times
T^{n+1}))$ are the colimits of complexes $C_*\ZF^{qf}_m(-,X\times
T^{n+1}))$ and $C_*\ZF_m(-,X\times T^{n+1}))$ over the suspension
morphisms $\Sigma$. The morphisms $In$ commute with the suspension
morphisms $\Sigma$, i.e. $\Sigma \circ In=In\circ \Sigma:
C_*\ZF^{qf}_m(-,X\times T^{n+1})\to C_*\ZF_{m+1}(-,X\times
T^{n+1})$. Proposition~\ref{ZF_qf_and_ZF} now completes the proof.
\end{proof}

\section{Proof of Theorem~\ref{cone}}\label{sec:cone}

In this section Theorem~\ref{cone} is proven, which is the main result of the paper. To prove the theorem, we
need further definitions as well as Theorem~\ref{ZM_fr_and_LM_fr} stated in the Introduction.
Let $\Gamma^{\op}$ be the category of finite pointed sets and
pointed maps. In what follows we shall identify $\Gamma^{\op}$ with a full
subcategory of $\Fr_0(k)$. The identification sends
$(K,*)\in\Gamma^{\op}$ to the non-pointed scheme $\spec k\sqcup\ldots\sqcup\spec k$,
where the coproduct is indexed by the set $K'=K\setminus *$.
In turn, $\Fr_0(k)$ is a full subcategory of $SmOp(\Fr_0(k))$. Thus $\Gamma^{\op}$ is a full
subcategory of $SmOp(\Fr_0(k))$. For each object $(Y,Y-S)$ in $SmOp(\Fr_0(k))$ and each finite pointed set $(K,*)$,
we write $(Y,Y-S)\otimes K$ to denote $(Y\times K',(Y-S)\times K')$ in $SmOp(\Fr_0(k))$.
In particular, if $Y\in \Fr_0(k)$, then $Y\otimes K=Y\times K'=Y\sqcup\ldots\sqcup Y$ with the
coproduct indexed by the elements of $K'$. This notation is consistent with that in~\cite[Section 5]{GP1}.

Following Notation~\ref{para}, let us define several $\Gamma$-spaces.
Namely, if $U,X\in \Fr_0(k)$ and $m\geq 0$ is an integer, consider the following $\Gamma$-spaces:
\begin{equation}\label{eq:Fr_n_U_X_times_T_m_otimes_K}
(K,*)\mapsto \Fr_n(U,(X\times T^m)\otimes K),
\end{equation}
\begin{equation*}\label{eq:ZF_n_U_X_times_T_m_otimes_K}
(K,*)\mapsto \ZF_n(U,(X\times T^m)\otimes K),
\end{equation*}
\begin{equation}\label{eq:F_n_U_X_times_T_m_otimes_K}
(K,*)\mapsto \F_n(U,(X\times T^m)\otimes K),
\end{equation}
where the right hand side pointed sets correspond to the values of the relevant
functors at the pair $(X\times(\bb A^1,\bb G_m)^{\wedge m})\otimes K=(X\times\A^m,X\times(\A^m-\{0\}))\otimes K$.
By Definition~\ref{F_m_U_Y/(Y-S)} the set $\F_n(U,(X\times T^m)\otimes K)-0_n$
is a free basis of the abelian group $\ZF_n(U,(X\times T^m)\otimes K)$.

Also, consider $\Gamma$-spaces
\begin{equation*}\label{eq:ZFr_n_U_X_times_T_m_otimes_K}
(K,*)\mapsto \ZFr_*(U,(X\times T^n)\otimes K),
\end{equation*}
\begin{equation*}
(K,*)\mapsto \ZF_*(U,(X\times T^n)\otimes K).
\end{equation*}
The associated presheaves of Segal $S^1$-spectra
will be denoted by $\bb Z\Fr_*^{S^1}(X\times T^n)$ and $EM(\ZF_*(-,X\times T^n))$, respectively. Thus,
   $$\bb Z\Fr_*^{S^1}(X\times T^n)=(\ZFr_*(-,X\times T^n),\ZFr_*(-,(X\times T^n)\otimes S^1),\ldots).$$
and
   $$EM(\ZF_*(-,X\times T^n))=(\ZF_*(-,X\times T^n),\ZF_*(-,(X\times T^n)\otimes S^1),\ldots).$$
The equality $\ZF_*(-,(X\sqcup X')\times T^n)=\ZF_*(-,X\times
T^n)\oplus \ZF_*(-,X'\times T^n)$ implies that the $\Gamma$-space
$(K,*)\mapsto \ZF_*(U,(X\times T^n)\otimes K)$ is fully determined
by the abelian group $\ZF_*(U,X\times T^n)$. Hence
$EM(\ZF_*(-,X\times T^n))$ is the Eilenberg--Mac~Lane spectrum for
$\ZF_*(-,X\times T^n)$.

The morphism of $\Gamma$-spaces $[(K,*)\mapsto
\ZFr_*(-,(X\times T^n)\otimes K)]\to [(K,*)\mapsto
\ZF_*(-,(X\times T^n)\otimes K)]$ induces a morphism of framed
$S^1$-spectra
   $$\lambda_{X\times T^n}:\bb Z\Fr_*^{S^1}(X\times T^n)\to EM(\ZF_*(-,X\times T^n)).$$
Also, denote by $\bb ZM_{fr}(X\times T^n)$, $X\in Sm/k$, the Segal
$S^1$-spectrum
   \begin{equation}\label{eq:ZMFr}
   (C_*\ZFr(-,X\times T^n),C_*\ZFr(-,(X\times T^n)\otimes S^1),\ldots).
   \end{equation}
Let $LM_{fr}(X\times T^n)$ be the Segal $S^1$-spectrum
   $$EM(\ZF(\Delta^\bullet\times-,X\times T^n))=(\ZF(\Delta^\bullet\times-,X\times T^n),\ZF(\Delta^\bullet\times-,(X\times T^n)\otimes S^1),\ldots).$$
The above arguments show that $LM_{fr}(X\times T^n)$ is the
Eilenberg--Mac~Lane spectrum associated with the complex
$\ZF(\Delta^\bullet\times-,X\times T^n)$.
The $\Gamma$-space morphism
$$[(K,*)\mapsto
\ZFr(\Delta^\bullet\times-,(X\times T^n)\otimes K)]\to [(K,*)\mapsto
\ZF(\Delta^\bullet\times-,(X\times T^n)\otimes K)]$$
induces a morphism of framed
$S^1$-spectra
   $$l_{X\times T^n}: \bb ZM_{fr}(X\times T^n)\to LM_{fr}(X\times T^n).$$
Note that stable homotopy groups of $LM_{fr}(X\times
T^n)=EM(\ZF(\Delta^\bullet\times-,X\times T^n))$ are equal to
homology groups of the complex $\ZF(\Delta^\bullet\times-,X\times
T^n)$. By~\cite[\S II.6.2]{Sch} homotopy groups $\pi_*(\bb
ZM_{fr}(X\times T^n)(U))$ of the $S^1$-spectrum
 $\bb ZM_{fr}(X\times T^n)$ evaluated at
$U\in \Fr_0(k)$ are homology groups $H_*(M_{fr}(X\times T^n)(U))$ of
$M_{fr}(X\times T^n)(U)$.

For the convenience of the reader we recall
Theorem~\ref{ZM_fr_and_LM_fr} proved in Appendix~B. It computes, in
particular, homology of the framed motive $M_{fr}(X\times T^n)$ of
the relative motivic sphere $X\times\A^n/(X\times(\A^n-\{0\}))$.

\begin{theorem*}
For any integer $m\geq 0$, the natural morphism of framed
$S^1$-spectra
   $$\lambda_{X\times T^m}:\bb Z\Fr_*^{S^1}(X\times T^m)\to EM(\ZF_*(-,X\times T^m))$$
is a schemewise stable equivalence. Moreover,
the natural morphism of framed $S^1$-spectra
   $$l_{X\times T^m}: \bb ZM_{fr}(X\times T^m)\to LM_{fr}(X\times T^m)$$
is a schemewise stable equivalence. In particular, for any $U\in
Sm/k$ one has
   $$\pi_*(\bb ZM_{fr}(X\times T^m)(U))=H_*(\ZF(\Delta^\bullet\times U,X\times T^m))=H_*(C_*\bb Z\F(U,X\times T^m)).$$
\end{theorem*}

As above, we can define the following $\Gamma$-space:
   $$(K,*)\mapsto \Fr(\Delta^\bullet\times-,(X\times T^m)\otimes K).$$
The Segal $S^1$-spectrum associated to this $\Gamma$-space is denoted by $M_{fr}(X\times T^m)$.
By construction,
   $$M_{fr}(X\times T^m)=(C_*\Fr(-,X\times T^m),C_*\Fr(-,(X\times T^m)\otimes S^1),\ldots).$$

We are now in a position to prove the remaining Theorem~\ref{cone}.

\begin{proof}[Proof of Theorem~\ref{cone}]
By Theorem~\ref{th:Main} the map of complexes of presheaves of Abelian groups
\begin{equation}\label{eq: C_*_A1/Gm}
C_*\bb Z\F(X\times T^{n}\times \A^1)/C_*\bb Z\F(X\times
T^{n}\times  \Gm) \to C_*\bb Z\F(X\times T^{n+1})
\end{equation}
is a local quasi-isomorphism. The $S^1$-spectra $LM_{fr}(X\times
T^{n}\times \A^1)$, $LM_{fr}(X\times T^{n}\times  \Gm)$ and
$LM_{fr}(X\times T^{n+1})$ are the Eilenberg--Maclane $S^1$-spectra
of the complexes $C_*\bb Z\F(X\times T^{n}\times \A^1)$, $C_*\bb
Z\F(X\times T^{n}\times \Gm)$ and $C_*\bb Z\F(X\times
T^{n+1})$ respectively. Thus the map
\begin{equation*}\label{eq: LM_fr_A1/Gm}
LM_{fr}(X\times T^{n}\times \A^1)/LM_{fr}(X\times T^{n}\times  \Gm)
\to LM_{fr}(X\times T^{n+1}),
\end{equation*}
induced by~\eqref{eq: C_*_A1/Gm}, is a local stable weak
equivalence, and hence so is the map
   $$\bb ZM_{fr}(X\times T^{n}\times \A^1)/\bb ZM_{fr}(X\times T^{n}\times  \Gm) \to \bb ZM_{fr}(X\times T^{n+1})$$
by Theorem~\ref{ZM_fr_and_LM_fr}. The $S^1$-spectra $M_{fr}(X\times T^{n}\times
\A^1)$, $M_{fr}(X\times T^{n}\times \Gm)$, $M_{fr}(X\times T^{n+1})$
are connected. Now the stable Whitehead theorem~\cite[II.6.30]{Sch}
implies the map
\begin{equation}\label{eq: M_fr_A1/Gm}
M_{fr}(X\times T^{n}\times \A^1)/M_{fr}(X\times T^{n}\times  \Gm) \to M_{fr}(X\times T^{n+1})
\end{equation}
is a local stable weak equivalence. It follows that the sequence of $S^1$-spectra
   $$M_{fr}(X \times T^n \times \mathbb G_m) \to M_{fr}(X \times T^n \times\bb A^1) \to M_{fr}(X \times T^{n+1})$$
is locally a homotopy cofiber sequence. This proves the second assertion of the theorem.

Next, prove that for any ${\ell} \geq 1$ the canonical morphism
\begin{equation}\label{eq:cone_and_T}
M_{fr}(\id_X\times \alpha^{\wedge\ell}): M_{fr}(X\times (\A^1//\mathbb G_m)^{\wedge\ell})\to
M_{fr}(X\times T^{\ell})
\end{equation}
is a local stable weak equivalence. As above, this reduces to verifying that the canonical morphism
of complexes of presheaves
   $$C_*\ZF(\id_X\times \alpha^{\wedge\ell}): C_*\ZF(X\times (\A^1//\mathbb G_m)^{\wedge\ell})\to C_*\ZF(X\times T^{\ell})$$
is locally a quasi-isomorphism.
We verify this using induction by $\ell$. First,
take $\ell=1$.
Let
$\mu: \ZF(\id_X\times \A^1//\mathbb G_m)\to \ZF(X\times \A^1)/\ZF(X\times \Gm)$
be the canonical morphism. Then $C_*(\mu)$ is a sectionwise quasi-isomorphism. Also,
$C_*\ZF(\id_X\times \alpha)=\tau \circ C_*(\mu)$
with $\tau$ as in Theorem \ref{th:Main}.
Thus by Theorem \ref{th:Main} $C_*\ZF(\id_X\times \alpha)$
is locally a quasi-isomorphism. We have verified the induction base case.
To do the inductiion step from $\ell-1$ to $\ell$,
note that
   $$C_*\ZF(\id_X\times \alpha^{\wedge\ell})=C_*\ZF(\id\times\alpha)\circ
       C_*\ZF(\id_X\times \alpha^{\wedge(\ell-1)}\times\id_{\A^1//\mathbb G_m}).$$
The map
$ C_*\ZF(\id\times\alpha)$
is locally a quasi-isomorphism by the case $\ell=1$.
The morphism
$C_*\ZF(\id_X\times \alpha^{\wedge(\ell-1)}\times\id_{\A^1//\mathbb G_m})$
is locally a quasi-isomorphism by the induction assumption.
We see that
$C_*\ZF(\id_X\times \alpha^{\wedge\ell})$
is locally a quasi-isomorphism, and hence
$M_{fr}(\id_X\times \alpha^{\wedge\ell})$ is
a local stable weak equivalence for each integer $\ell \geq 1$.

It remains to prove that for each integer $\ell\geq 1$ the
morphism~\eqref{eq:cone_and_T} is a local level weak equivalence.
By the Additivity Theorem of~\cite{GP1} and
Lemma~\ref{l:connectivity}
the $S^1$-spectrum $M_{fr}(X\times T^{\ell})(U)$ with $U$ a local Henselian
smooth scheme and $X$ any smooth scheme is a
connected $\Omega$-spectrum. By the Additivity Theorem of~\cite{GP1} the
$\Gamma$-space $K \to C_*\Fr(U,(X\times (\A^1//
\mathbb G_m)^{\wedge \ell}) \otimes K)$ is special. By Lemma~\ref{l:connectivity} the zeroth
space $C_*\Fr(-,X\times (\A^1//\mathbb G_m)^{\wedge \ell})$ of
the Segal $S^1$-spectrum $M_{fr}(X\times (\A^1// \mathbb
G_m)^{\wedge\ell})(U)$ is connected. Thus the Segal $S^1$-spectrum
$M_{fr}(X\times (\A^1// \mathbb G_m)^{\wedge\ell})(U)$ is a
connected $\Omega$-spectrum. Since the
morphism~\eqref{eq:cone_and_T} is locally a stable weak equivalence between $\Omega$-spectra,
it is a local level weak equivalence.
The theorem is proved.
\end{proof}

\begin{corollary}\label{cormain}
Let $k$ be an infinite perfect field.
For every $n\geq 0$, the natural morphism
$$M_{fr}(X\times T^n\times(\A^1// \mathbb G_m))\to M_{fr}(X\times T^{n+1})$$
is locally a level weak equivalence of $S^1$-spectra.
\end{corollary}

\begin{proof}
Consider a commutative diagram
   $$\xymatrix{M_{fr}(X\times(\A^1//\mathbb G_m)^{\wedge n+1})\ar[d]\ar[r]&M_{fr}(X\times T^{n+1})\ar@{=}[d]\\
               M_{fr}(X\times T^{n}\times(\A^1//\mathbb G_m))\ar[r]&M_{fr}(X\times T^{n+1})}$$
The left vertical and upper horizontal arrows are locally level weak
equivalences of $S^1$-spectra by Theorem~\ref{cone}, and hence so is
the lower horizontal arrow.
\end{proof}

\appendix\section{}

In this section we prove the following useful

\begin{lemma}\label{l:connectivity}
For any $X\in Sm/k$ and any $n>0$ the simplicial pointed presheaves
$C_*\Fr(-,X\times (\A^1//\mathbb G_m)^{\wedge n})$ and
$C_*\Fr(-,X\times T^n)$ are locally connected in the Nisnevich
topology.
\end{lemma}

\begin{proof}
Firstly check the Nisnevich local connectivity of $C_*\Fr(-,X\times
(\A^1//\mathbb G_m))$. Clearly, the map $\pi_0(C_*\Fr(-,X\times
\A^1))\to\pi_0(C_*\Fr(-,X\times(\A^1//\mathbb G_m)))$ is
surjective. On the other hand the composite map of pointed sets
   $$\pi_0(C_*\Fr(-,X\times \Gm))\to\pi_0(C_*\Fr(-,X\times\A^1))\to\pi_0(C_*\Fr(-,X\times(\A^1//\mathbb G_m)))$$
is constant, because it factors through the pointed set
$\pi_0(C_*\Fr(-,(X\times \Gm)\otimes I))=*$. Thus it is sufficient
to check that for any local essentially $k$-smooth Henselian $U$ the
map
\begin{equation}\label{eq:connectivity}
\pi_0(C_*\Fr(U,X\times \Gm))  \to  \pi_0(C_*\Fr(U,X\times \A^1))
\end{equation}
is surjective. Take a framed correspondence
   $$c_0=(Z,W, \phi; (f,g):W\to X\times \A^1)\in \Fr_n(U,X\times \A^1).$$
We may now assume that $W=(\A^n_U)^h_Z$ is the henselization of
$\A^n_U$ at the closed subset $Z$.
Since the scheme $\A^n_U$ is affine Noetherian, $(\A^n_U)^h_Z$ is an affine Noetherian scheme
(see \cite[6.9]{G}).
We want to find $h_t\in
\Fr_n(\A^1_U,X\times \A^1)$ and $c_1\in \Fr_n(U,X\times \Gm_{,k})$ such
that $h_0=c_0$, $h_{1}=j \circ c_1$, where $j:X\times \Gm
\hookrightarrow X\times \A^1_k$  is the open embedding.

Construct now a required element $h_t\in \Fr_n(\A^1_U,X\times \A^1)$.
To do that consider a closed subset
$\tilde Z$ in $\A^1\times Z$ defined by the equation
$t=g|_Z$. Clearly, $\tilde Z$ is isomorphic to $Z$. Hence
it is finite over $U$. Since the field $k$ is infinite and $\tilde Z$ is finite over $U$, there is a
non-zero element $a\in k$ such that
$(a\times Z)\cap \tilde Z=\emptyset$ in $\A^1\times Z$. We may assume that $a=1$.
Put
   $$h_t=(\tilde Z,\A^1\times W,\phi_1,...,\phi_{n},g-t;f:W\to X)\in\Fr_n(\A^1_U,X\times T).$$
Clearly, $h_0=b_1$ and $h_{1}=j \circ c_1$ for certain $c_1\in \Fr_n(U,X\times \Gm_{,k})$ .
Thus
the arrow (\ref{eq:connectivity}) is surjective and
$C_*\Fr(-,X\times (\A^1//\mathbb G_m))$ is locally connected.

By induction, suppose $C_*\Fr(-,X\times (\A^1//\mathbb
G_m)^{\wedge n})$ is locally connected. Then $C_*\Fr(-,X\times
(\A^1//\mathbb G_m)^{\wedge(n+1)})$ is the realization of a
simplicial space of the form
   $$[r]\mapsto C_*\Fr(-,Y_r\times(\A^1//\mathbb G_m)^{\wedge n})$$
with $Y_r\in Sm/k$. Since each $C_*\Fr(-,Y_r\times
(\A^1//\mathbb G_m)^{\wedge n})$ is locally connected by the
induction hypothesis, then $C_*\Fr(-,X\times (\A^1//\mathbb
G_m)^{\wedge(n+1)})$ is locally connected as well.

Now let us prove the Nisnevich local connectivity of
$C_*\Fr(-,X\times T^n)$. We give a proof for $n=1$. The general
case is treated similarly. For any local essentially $k$-smooth
Henselian $U$ consider a framed correspondence
$b_0=(Z,V, \phi_1,...,\phi_n,\phi_{n+1}; f:V\to X)\in \Fr_n(U,X\times T)$.
It is sufficient to find a family of elements
\[\{h^{(0)}_t, h^{(1)}_t,\dots ,h^{(r)}_t\} \subset \Fr_n(\A^1_U,X\times T)\]
such
that $h^{(0)}_0=b_0$, $h^{(i)}_1=h^{(i+1)}_0$ for $i=0,\dots ,r-1$ and $h^{(r)}_{1}=0_n$ is the empty framed correspondence.
By Lemma \ref{l:homotopy_properties}
we may assume that
$b_1:=h^{(0)}_1$ is in
$\Fr^{qf}_n(\A^1_U,X\times T)$.
Write $b_1$ as
$b_1=(Z,W, \psi_1,...,\psi_n,\psi_{n+1}; g:W\to X)$.
We may assume that $W=(\A^n_U)^h_Z$ is the henselization of
$\A^n_U$ at the closed subset $Z$. Since
$b_1$ is in
$\Fr^{qf}_n(\A^1_U,X\times T)$
the closed subset
$Y:=\{\psi_1=\psi_2=...=\psi_n=0\}$ of $W$
is quasi-finite over $U$.
The following lemma shows that we may assume
$Y$ is finite over $U$.

\begin{lemma}\label{l: qf_to_finite}
There is an affine Zariski open $W^0\subset W$ such that the closed
subset $Y^0:=W^0\cap Y$ is finite over $U$, $Z$ is contained in
$W^0$ and
$(Z,W^0,\psi_1,...,\psi_n,\psi_{n+1};g)=(Z,W,\psi_1,...,\psi_n,\psi_{n+1};g)$.
\end{lemma}

\begin{proof}
Before proving the lemma note that its first two assertions yield
the last one. To prove the lemma, consider a closed subset $Y'$ of
$Y$ which is the union of all connected components $Y'_i$ of $Y$
having non-empty intersection with $Z$. Let $Y''$ be the union of
all other connected components of $Y$. Clearly, $Y=Y'\sqcup Y''$.
Put $W^0=W-Y''$ and $Y^0=W^0\cap Y$. Then $Y'=Y^0$. It remains to
check that $Y'$ is finite over the local henselian scheme $U$. It is
sufficient to check that each $Y'_i$ is finite over $U$. Put
$S_i=Y'_i\cap Z$. Then $S_i$ is a non-empty closed subset of $Z$. So
$S_i$ is non-empty and finite over $U$. Thus the fibre $S_{i,u}$
over the closed point $u\in U$ is non-empty, and hence so is the
fibre $Y'_{i,u}$ of $Y'_i$ over $u$. Now~\cite[Theorem I.4.2]{Mi}
implies finiteness of $Y'_i$ over $U$. Lemma~\ref{l: qf_to_finite}
is proved.
\end{proof}

Construct now an element $H_t\in \Fr_n(\A^1_U,X\times T)$
such that $H_0=b_1$ and $H_{1}=0_n$.
Consider a closed subset
$\tilde Y$ in $\A^1\times Y$ defined by the equation
$t=\psi_{n+1}|_Y$. Clearly, $\tilde Y$ is isomorphic to $Y$. Hence
it is finite over $U$. Since the field $k$ is infinite and $\tilde Y$ is finite over $U$, there is a
non-zero element $a\in k$ such that
$(a\times Y)\cap \tilde Y=\emptyset$ in $\A^1\times Y$.
We may assume that $a=1$.
Put
   $$H_t=(\tilde Y,\A^1\times W,\psi_1,...,\psi_{n},\psi_{n+1}-t;g:W\to X)\in\Fr_n(\A^1_U,X\times T).$$
Clearly, $H_0=b_1$ and $H=0_n$.
Thus $C_*\Fr(-,X\times T)$ is locally connected. The same
is true for $C_*\Fr(-,X\times T^n)$. This proves Lemma~\ref{l:connectivity}.

\end{proof}

\section{}

The main goal of this section is to prove Theorem~\ref{ZM_fr_and_LM_fr}.
It will be proved at the end of the section.

Let $\bb S$ be the sphere $S^1$-spectrum. Let $* \subset \bb S$ be
its trivial $S^1$-subspectrum corresponding to the basepoint. Let $\mathcal A$ be a pointed set
with a distinguished point $*$. Denote by $\bb S_{\mathcal A}$ the $S^1$-spectrum
$\prod_{(\mathcal A-*)}\bb S$. Let $\bb S'_{\mathcal A}$ be the
$S^1$-subspectrum $\vee_{(\mathcal A-*)}\bb S$ in $\bb S_{\mathcal A}$.

Given a finite pointed subset $A\subset \mathcal A$, let $\bb S_{A}\subset \bb
S_{\mathcal A}$ be an $S^1$-subspectrum of the form $\prod_{\mathcal
(A-*)}E_a$, where $E_a=\bb S$ if $a\in A-*$ and $E_a=*$ if $a\in \mathcal
A-A$. If $a\in A-*$ we shall write $\bb S_a$ to denote $\bb S_{\{a,*\}}$, where
$\{a,*\}\subset A$ is the two elements subset of $A$. Let $\bb
S'_A\subset \bb S_A$ be the $S^1$-subspectrum $\vee_{a\in (A-*)} \bb
S_a$ in $\bb S_A$. Clearly, the inclusion $\bb S'_A\subset \bb S_A$
is a stable equivalence of $S^1$-spectra.
Set $\bb S^{f}_{\mathcal A}:=\cup_{A\subset \mathcal A}\bb S_A\subset \bb S_{\mathcal A}$,
where the union is taken over the set of all finite pointed subsets $A$ of the pointed set
$\mathcal A$.

The following lemma is straightforward and the proof is left to the reader.

\begin{lemma}\label{l:cup_S_A_and_cup_S'_A}
Let $\mathcal A$ be a pointed set and $I$ be a set such that
for any $i\in I$ there is a finite pointed subset $A(i)\subset \mathcal A$.
We have two $S^1$-subspectra $\cup_{i\in I}\bb S'_{A(i)}$,
$\cup_{i\in I}\bb S_{A(i)}$ of the spectrum $\bb S^f_{\mathcal A}$.
Suppose $\cup_{i\in I}A(i)=\mathcal A$. Then
$\vee_{(\mathcal A-*)}\bb S_a=\cup_{i\in I}\bb S'_{A(i)}$ and the inclusion
   $$\vee_{(\mathcal A-*)}\bb S_a=\cup_{i\in I}\bb S'_{A(i)} \hookrightarrow \cup_{i\in I}\bb S_{A(i)}$$
is a stable equivalence of $S^1$-spectra.
\end{lemma}

An application of this lemma is given below in this section. For a
finite pointed set $(K,*)$ consider a set $Map^f_\bullet(\mathcal
A,K)$ of those maps $\rho$ of pointed sets such that the set
$\rho^{-1}(K-*)$ is finite. Consider a $\Gamma$-space
$\Gamma^f_{\mathcal A}$ defined by $\Gamma^f_{\mathcal
A}(K,*)=Map^f_\bullet(\mathcal A,K)$. For a finite pointed subset
$A\subset \mathcal A$ consider a subset $Map^A_\bullet(\mathcal
A,K)\subset Map^f_\bullet(\mathcal A,K)$ consisting of all maps
$\rho$ such that $\rho^{-1}(K-*)\subset A$. Consider a
$\Gamma$-space $\Gamma_{A}$ defined by
$\Gamma_{A}(K,*)=Map^A_\bullet(\mathcal A,K)$.

Clearly, for any inclusion of finite pointed subsets $A'\subset A$ of $\mathcal A$ one has inclusions
$Map^{A'}_\bullet(\mathcal A,K)\subset Map^A_\bullet(\mathcal A,K)$
and $\Gamma_{A'}\subset \Gamma_{A}$. Moreover, one has
   $$\cup_{A\subset \mathcal A}Map^A_\bullet(\mathcal A,K)=Map^f_\bullet(\mathcal A,K) \ \text{and} \
\cup_{A\subset \mathcal A}\Gamma_{A}=\Gamma^f_{\mathcal A},$$
where the union is taken over the set of all pointed finite subsets $A$ in the pointed set
$\mathcal A$.

Let $A\subset \mathcal A$ be a finite pointed subset. For any element
$a\in A-*$ set $Map^a_\bullet(\mathcal A,K)=Map^{a\sqcup *}_\bullet(\mathcal A,K)$,
where $a\sqcup *$ stands for the two elements pointed subset of $\mathcal A$.
Set $\Gamma_a=\Gamma_{a\sqcup *}$. That is
$\Gamma_{a}(K,*)=Map^{a\sqcup *}_\bullet(\mathcal A,K)$.
Let $Map^{A,s}_\bullet(\mathcal A,K)\subset Map^A_\bullet(\mathcal A,K)$
consist of maps $\rho$
such that the subset $\rho^{-1}(K-*)\subset A$ either has one element or is the empty set.
Consider a $\Gamma$-subspace $\Gamma'_{A}\subset \Gamma_{A}$ such that
$\Gamma'_{A}(K,*)=Map^{A,s}_\bullet(\mathcal A,K)$.

The $\Gamma$-space $\Gamma_{A}$ is isomorphic to the $\Gamma$-space
$\prod_{a\in (A-*)}\Gamma_a$. The $\Gamma$-space $\Gamma'_{A}$ is isomorphic to the $\Gamma$-space $\vee_{a\in (A-*)}\Gamma'_a$. Moreover these isomorphisms are consistent with the inclusion $\vee_{a\in (A-*)}\Gamma_a \subset \prod_{a\in (A-*)}\Gamma_a$.

\begin{lemma}\label{l:cup_G_A_and_cup_G'_A}
Let $\mathcal A$, $I$ and $A(i)$ be as in Lemma
\ref{l:cup_S_A_and_cup_S'_A}.
There are two $\Gamma$-subspaces $\cup_{i\in I}\Gamma'_{A(i)}$,
$\cup_{i\in I}\Gamma_{A(i)}$ of the $\Gamma$-space $\Gamma^f_{\mathcal A}$.
Suppose $\cup_{i\in I}A(i)=\mathcal A$. Then
$\vee_{a \in (\mathcal A-*)}\Gamma_a=\cup_{i\in I}\Gamma'_{A(i)}$ and we have natural inclusions
$$\vee_{a\in (\mathcal A-*)}\Gamma_a=\cup_{i\in I}\Gamma'_{A(i)} \hookrightarrow \cup_{i\in I}\Gamma_{A(i)}$$
of $\Gamma$-spaces.
\end{lemma}

Let $Seg: \Gamma-spaces \to S^1-spectra$ be the functor associating the Segal $S^1$-spectrum
to a $\Gamma$-space. Then $Seg(\Gamma^f_{\mathcal A})=\bb S^f_{\mathcal A}$.
Given a non-distinguished element $a\in \mathcal A$ one has $Seg(\Gamma_a)=\bb S$
and the functor $Seg$ converts the inclusion
$\Gamma_a\subset \Gamma^f_{\mathcal A}$ to the inclusion
$\bb S_a\subset \bb S^f_{\mathcal A}$.
For any finite pointed subset $A$ in $\mathcal A$, the functor $Seg$ converts the inclusion
$\Gamma_A\subset \Gamma^f_{\mathcal A}$
to the inclusion $\bb S_A\subset \bb S^f_{\mathcal A}$.
It also converts the inclusion $\Gamma'_A\subset \Gamma_A$
to the inclusion $\bb S'_A\subset \bb S_A$ as well as the inclusion
$\vee_{a \in (\mathcal A-*)}\Gamma_a\subset \Gamma^f_{\mathcal A}$
to the inclusion
$\vee_{a \in (\mathcal A-*)}\bb S_a\subset \bb S^f_{\mathcal A}$.

The above arguments together with Lemmas
\ref{l:cup_S_A_and_cup_S'_A} and \ref{l:cup_G_A_and_cup_G'_A} prove
the following

\begin{lemma}\label{l:cup_Seg_A_and_cup_Seg'_A}
Let $\mathcal A$, $I$ and $A(i)$ be as in Lemma
\ref{l:cup_S_A_and_cup_S'_A}.
There are two $S^1$-subspectra $\cup_{i\in I}Seg(\Gamma'_{A(i)})$,
$\cup_{i\in I}Seg(\Gamma_{A(i)})$ of the $S^1$-spectrum $Seg(\Gamma^f_{\mathcal A})$.
Suppose $\cup_{i\in I}A(i)=\mathcal A$. Then
$\vee_{a \in (\mathcal A-*)}Seg(\Gamma_a)=\cup_{i\in I}Seg(\Gamma'_{A(i)})$ and the inclusion
$$\vee_{a\in (\mathcal A-*)}Seg(\Gamma_a)=\cup_{i\in I}Seg(\Gamma'_{A(i)}) \hookrightarrow \cup_{i\in I}Seg(\Gamma_{A(i)})$$
is a stable equivalence of $S^1$-spectra.
\end{lemma}

\begin{notation}\label{n:A_and_I_specific}{\rm
Let $U,X\in Sm/k$ and let $m,n\geq 0$ be integers. Set $\mathcal
A=\F_m(U,X\times T^n)$ and regard it as a pointed set pointed by
the empty framed correspondence $0_m$. Set $I=\Fr_m(U,X\times
T^n)-0_m$. }\end{notation}

In the remaining part of this section we use notation from Section
\ref{s:Fr(X_T_n}.

\begin{definition}\label{def:A(Phi)}{\rm
Given $\Phi=(Z,W,\phi;g),\Phi'=(Z',W',\phi';g')\in \Fr_m(U,X\times
T^n)$, we write $\Phi'\leq \Phi$ if there is a closed subset $Z''$
in $\A^m\times U$ such that $Z=Z'\sqcup Z''$ and
   $$(Z',W',\phi';g')=(Z',W-Z'',\phi|_{W-Z''};g|_{W-Z''}) \in \Fr_m(U,X\times T^n).$$
For any $\Phi\in \Fr_m(U,X\times T^n)$ set
$A(\Phi)=\{\Psi\in \F_m(U,X\times T^n): \Psi \leq \Phi \} \subset \mathcal A$.
Clearly,
\begin{equation*}\label{eq:A_and_A_Phi}
\mathcal A=\cup_{\Phi\in I} A(\Phi).
\end{equation*}
}
\end{definition}

For a finite pointed set $(K,*)$ the pointed set $\Fr_m(U,(X\times
T^n)\otimes K)$ is defined by the
formula~\eqref{eq:Fr_n_U_X_times_T_m_otimes_K}. Let $K'=K-*$. By
Definition~\ref{def:FrY/Y-S}(III) the set $\Fr_m(U,(X\times
T^n)\otimes K)$ consists of equivalence classes of tuples
$(Z,W,\phi;g;f)$, where $Z$ is a closed subset of $U\times\bb A^m$,
finite over $U$, $W$ is an \'{e}tale neighborhood of $Z$ in
$U\times\bb A^m$,
$\phi_1,\ldots,\phi_{m},\phi_{m+1},\ldots,\phi_{m+n}$ are regular
functions on $W$, $(g,f): W\to X\times K'$ is a regular map such
that $Z=Z(\phi_1,\ldots,\phi_{m+n})$. Notice that regular maps from
$W$ to $X\otimes K$ are in one-to-one correspondence with couples of
regular maps $(W\to X,W\to K')$.

For a finite pointed set $(K,*)$, the pointed set $\F_m(U,(X\times
T^n)\otimes K)$ is defined by the
formula~\eqref{eq:F_n_U_X_times_T_m_otimes_K}. By
Definition~\ref{F_m_U_Y/(Y-S)}
it consists of those elements
$(Z,W,\phi;g;f)\in \Fr_m(U,(X\times T^n)\otimes K)$ such that the
closed subset $Z$ of $U\times\bb A^m$ is connected.

\begin{definition}\label{def:Gamma_m_and_Gamma'_m}{\rm
Denote by $\Gamma_m(U,X\times T^n)$ the $\Gamma$-space $(K,*)\mapsto
\Fr_m(U,(X\times T^n)\otimes K)$. \\
Similarly, $\Gamma'_m(U,X\times T^n)$ stands
for the $\Gamma$-space $(K,*)\mapsto \F_m(U,(X\times T^n)\otimes
K)$.

Given $\Phi\in I$ define $\Gamma_m(U,X\times T^n)_{\Phi}$ as a
$\Gamma$-subspace of the $\Gamma$-space $\Gamma_m(U,X\times T^n)$
such that for a finite pointed set $(K,*)$
   $$\Gamma_m(U,X\times T^n)_{\Phi}(K)=\{(Z,W,\phi;g;f)\in \Fr_m(U,(X\times T^n)\otimes K)\mid (Z,W,\phi;g)\leq \Phi \in \Fr_m(U,X\times T^n)\}.$$
Define $\Gamma'_m(U,X\times T^n)_{\Phi}$ as a $\Gamma$-subspace of
the $\Gamma$-space $\Gamma'_m(U,X\times T^n)$ such that for a finite
pointed set $(K,*)$
   $$\Gamma'_m(U,X\times T^n)_{\Phi}(K)=\{(Z,W,\phi;g;f)\in \F_n(U,(X\times T^n)\otimes K)\mid (Z,W,\phi;g)\leq \Phi \in \F_m(U,X\times T^n)\}.$$

}\end{definition}

\begin{definition}{\rm
For a finite pointed set $(K,*)$ put $K'=K-*$ and consider a pointed
set map
   $$inc_{K}: \Fr_m(U,(X\times T^n)\otimes K)\to Map^f_{Sets_{\bullet}}(\mathcal A,K),$$
which is defined as follows. Let $\Psi=(Z,W,\phi;g;f)\in
\Fr_m(U,(X\times T^n)\otimes K)$ and $a=(Z_a,W_a,\phi_a;g_a)\in
\mathcal A=\F_m(U,X\times T^n)$. If the element $a$ is in
$\mathcal A-A((Z,W\phi;g))$, then the map $inc_{K}(\Psi)$ takes the
element $a$ to the distinguished point $*$ of the set $K$. If $a\in
A((Z,W,\phi;g))-0_m$, then the map $inc_{K}(\Psi)$ takes the element
$a$ to $f(Z_a)\in K'\subset K$. Finally, the map $inc_{K}(\Psi)$
sends $0_m$ to the distinguished point $*$ of the set $K$.

Recall that $Z_a$ is connected and if $a\in A((Z,W,\phi;g))$, then
$Z=Z_a\sqcup Z''$ for some $Z''$. Define a $\Gamma$-space morphism
$$inc_m: \Gamma_m(U,X\times T^n)\to \Gamma^f_{\mathcal A}$$
sending a finite pointed set $(K,*)$ to the pointed set map
$inc_{K}$. It is straightforward to check that it is indeed a
$\Gamma$-space morphism.

}\end{definition}

The following lemma is crucial.

\begin{lemma}\label{l:key}
The $\Gamma$-space morphism $inc_m$ is injective. Moreover, using
this inclusion the following identifications hold:

\begin{enumerate}
\item for any $\Phi\in I$, one has $\Gamma_m(U,X\times
T^n)_{\Phi}=\Gamma_{A(\Phi)}$ and $\cup_{\Phi\in
I}\Gamma_m(U,X\times T^n)_{\Phi}= \cup_{\Phi\in I}\Gamma_{A(\Phi)}$;
\item for any $\Phi\in I$, one has $\Gamma'_m(U,X\times
T^n)_{\Phi}=\Gamma'_{A(\Phi)}$ and $\cup_{\Phi\in
I}\Gamma'_m(U,X\times T^n)_{\Phi}=\cup_{\Phi\in
I}\Gamma'_{A(\Phi)}$;
\item for any $a\in \mathcal A-0_m=\F_m(U,X\times T^n)-0_m$, one has $\Gamma_m(U,X\times T^n)_a=\Gamma_a$;
\item $\vee_{a\in (\mathcal A-*)}\Gamma_m(U,X\times T^n)_a=\vee_{a\in (\mathcal A-*)}\Gamma_a$.
\end{enumerate}
\end{lemma}

Applying the Segal functor $Seg$, we see that
Lemmas~\ref{l:cup_G_A_and_cup_G'_A}
and~\ref{l:cup_Seg_A_and_cup_Seg'_A} imply the following

\begin{proposition}
Let $\mathcal A=\F_m(U,X\times T^n)$, $I=\Fr_m(U,X\times
T^n)-0_m$ be as Notation~\ref{n:A_and_I_specific} and for $\Phi\in
I$ let the subset $A(\Phi)\subset \mathcal A$ be as in
Definition~\ref{def:A(Phi)}. There are two $S^1$-subspectra
$\cup_{i\in I}Seg(\Gamma'_m(U,X\times T^n)_{\Phi})$, $\cup_{i\in
I}Seg(\Gamma_m(U,X\times T^n)_{\Phi})$ of the $S^1$-spectrum
$Seg(\Gamma^f_{\mathcal A})$. One has an equality of the
$S^1$-subspectra
   $$\vee_{a \in (\mathcal A-*)}Seg(\Gamma_m(U,X\times T^n)_a)=\cup_{\Phi\in I}Seg(\Gamma'_m(U,X\times T^n)_{\Phi})$$
and the inclusion
\begin{multline*}
\vee_{a\in (\mathcal A-*)}Seg(\Gamma_m(U,X\times T^n)_a)=\cup_{\Phi\in I}Seg(\Gamma'_m(U,X\times T^n)_{\Phi})=Seg(\Gamma'_m(U,X\times T^n)) \hookrightarrow \\
\hookrightarrow \cup_{\Phi\in I}Seg(\Gamma_m(U,X\times T^n)_{A(\Phi)})=Seg(\Gamma_m(U,X\times T^n))
\end{multline*}
is a stable equivalence of the $S^1$-spectra.
\end{proposition}

Set,
   $$\Fr^{S^1}_m(U,X\times T^n)=Seg(\Gamma_m(U,X\times T^n)) \ \text{and} \ \Fr^{S^1}_m(U,X\times T^n)_{\Phi}=Seg(\Gamma_m(U,X\times T^n)_{\Phi}),$$
   $$\F^{S^1}_m(U,X\times T^n)=Seg(\Gamma'_m(U,X\times T^n)) \ \text{and} \  \F^{S^1}_m(U,X\times T^n)_{\Phi}=Seg(\Gamma'_m(U,X\times T^n)_{\Phi}).$$

Under this notation the preceding proposition implies the following

\begin{theorem}\label{p:Fr_and_F}
Let $U,X\in Sm/k$ and let $m,n\geq 0$ be integers. One has an
equality of $S^1$-subspectra
$$\vee_{\Psi \in (\F_m(U,X\times T^n)-0_m)}\F^{S^1}_m(U,X\times T^n)_{\Psi}=\cup_{\Phi\in (\Fr_m(U,X\times T^n)-0_m)}\F^{S^1}_m(U,X\times T^n)_{\Phi}$$
of the spectra $\Fr^{S^1}_m(U,X\times T^n)$ and the inclusion
\begin{equation}\label{eq:F_S1_and_Fr_S1}
\F^{S^1}_m(U,X\times T^n)\subset \Fr^{S^1}_m(U,X\times T^n)
\end{equation}
is a stable equivalence of the $S^1$-spectra.
\end{theorem}

We are now in a position to prove Theorem~\ref{ZM_fr_and_LM_fr}.

\begin{proof}[Proof of Theorem~\ref{ZM_fr_and_LM_fr}]
This follows from Theorem~\ref{p:Fr_and_F}. Indeed, consider the
composite morphism of $S^1$-spectra
   $$\ZF^{S^1}_m(U,X\times T^n)\to \ZFr^{S^1}_m(U,X\times T^n)\xrightarrow{\lambda_{X\times T^n}} \ZF^{S^1}_m(U,X\times T^n),$$
where the left arrow is induced by the
arrow~\eqref{eq:F_S1_and_Fr_S1}. Within
Definitions~\ref{F_m_U_Y/(Y-S)} and~\ref{def:Gamma_m_and_Gamma'_m},
Theorem \ref{p:Fr_and_F} implies the left arrow is a stable
equivalence of $S^1$-spectra. Note that the composite morphism is
the identity map. Thus the morphism $\lambda_{X\times T^n}$ is a
stable equivalence of $S^1$-spectra. This finishes the proof.
\end{proof}

\end{document}